\documentclass[final,1p,times,authoryear]{elsarticle}

\usepackage{amsmath}
\usepackage{amsthm}
\usepackage{amssymb}
\usepackage{amsfonts}
\usepackage{graphicx}
\usepackage{algorithm}
\usepackage[noend]{algpseudocode}
\usepackage{url}
\usepackage[capitalise]{cleveref}
\usepackage{tikz}
\usepackage{graphicx}
\usepackage{wrapfig}
\usepackage{lscape}
\usepackage{rotating}
\usepackage{epstopdf}
\usepackage{showkeys}

\newtheorem{theorem}{Theorem}
\newtheorem{thm}[theorem]{Theorem}
\newtheorem{lem}[theorem]{Lemma}
\newtheorem{cor}[theorem]{Corollary}
\newtheorem{hyp}[theorem]{Hypothesis}

\newtheorem{definition}[theorem]{Definition}

\newtheorem{ex}[theorem]{Example}

\newtheorem{rem}[theorem]{Remark}

\newenvironment{roof}{\noindent\textit{Proof.}}
 {\hspace*{\fill}\tiny$\blacksquare$}

\newcommand{\OPart}{\textrm{OPart}}
\newcommand{\Sym}{\textrm{Sym}}
\newcommand{\Co}{\textrm{Co}}

\newcommand{\s}{\mathcal{S}}
\newcommand{\N}{\mathbb{N}}

\newcommand{\la}{\langle}
\newcommand{\ra}{\rangle}
\newcommand{\Fixed}{\textrm{Fixed}} %
\newcommand{\Orb}{\textrm{Orb}}     % Orbital graph refiner
\newcommand{\DeepOrb}{\textrm{DeepOrb}}     % Orbital graph refiner

\begin{document}

\begin{frontmatter}

\title{New refiners for permutation group search}

\author{Christopher Jefferson}
\ead{caj21@st-andrews.ac.uk}
\ead[url]{http://caj.host.cs.st-andrews.ac.uk/}

\author{Markus Pfeiffer}
\address{University of St~Andrews\\School of Computer Science\\North Haugh\\St Andrews\\KY16 9SX\\Scotland}
\ead{markus.pfeiffer@st-andrews.ac.uk}
\ead[url]{https://www.morphism.de/~markusp/}

\author{Rebecca Waldecker}
\address{Martin-Luther-Universit\"at Halle-Wittenberg\\Institut f\"ur Mathematik\\06099 Halle\\Germany}
\ead{rebecca.waldecker@mathematik.uni-halle.de}
\ead[url]{http://conway1.mathematik.uni-halle.de/~waldecker/index-english.html}

\begin{abstract}
Partition backtrack is the current generic state of the art algorithm to search for subgroups
of a given permutation group.
We describe an improvement of partition backtrack for set stabilizers and intersections
of subgroups by using orbital graphs. With extensive experiments
we demonstrate that our methods improve performance of partition backtrack -- in some cases by several
orders of magnitude.
\end{abstract}

\end{frontmatter}

\section{Introduction}
\label{sec:intro}

Permutation groups are one of the most natural and convenient representations
of finite groups. They have proved particularly useful for
computational purposes, and systems such as~\cite{GAP4} and Magma (\cite{MR1484478})
provide efficient implementations of many algorithms to solve
a range of problems from membership testing to identification of the isomorphism
type of a group.

Given a permutation group acting on a finite set \(\Omega\), some problems -- for
example checking whether or not a group contains a particular permutation, or computing the
size of the group -- can be solved in time polynomial in the size of
\(\Omega\). There are problems for which no polynomial time algorithm is known, for example computing intersections, centralizers and normalizers of subgroups,
and set and partition stabilizers. For these problems the best known algorithms perform a very sophisticated exhaustive search through the group in question.
It is known (see for example Chapter 3 in \cite{seress2003permutation}) that these
problems are at least as difficult as graph isomorphism, so it is unlikely that a
polynomial time algorithm can be found.

The current state of the art algorithm is described in \cite{Leon} and is commonly called
partition backtrack. It extends ideas introduced in \cite{McKay80} to  solve graph isomorphism problems. 
Implementations of partition backtrack are available in \cite{GAP4} and Magma (\cite{MR1484478}),
and they solve the problems mentioned above very efficiently for a fair range of examples.

This does not mean that all permutation group problems can be solved
easily or quickly using partition backtrack -- and this is where our orbital
graph methods come into play.
The ideas presented in this paper improve performance by several orders of
magnitude for a range of examples, as will be demonstrated in Section 6.
Partition backtrack should not be viewed as a monolithic algorithm, but rather
as a combination of algorithms solving particular sub-problems.
The concept of a \textbf{refiner} is one of its crucial components.
They are used to detect and skip parts of the computation that would be
superfluous: Better refiners lead to more superfluous computational steps to be
skipped.

We describe a new class of refiners using orbital graphs. It can be used to
improve performance of partition backtrack implementations, and we demonstrate the speed-up
that can be achieved in our implementation of partition backtrack.
This article is organized as follows:

Section~\ref{sec:notation} includes necessary notation and examples of permutation groups and
ordered partitions.
In Section~3 we give a brief description of backtrack search and the role of
refiners, and we discuss some standard refiners.
Section~4 introduces orbital graphs and gives a characterization of the cases where
orbital can benefit refinement, we briefly discuss the limits of refiners using orbital graphs: For example, for $2$-transitive groups, orbital graphs do not provide any improvement. 
Section~5 then combines ideas introduced in Sections~3 and 4 to define new refiners using orbital graphs.
Finally, we explain and present experiments and their outcomes in Section 6. In particular, there are cases in which the costs of calculating orbital graphs outweigh the advantages in the search,
but for some search problems where the previously best techniques perform poorly, our methods prove to be very effective.

At the end of the paper we comment on related questions and further research.

\textbf{Acknowledgements.}

All authors thank the DFG (\textbf{Wa 3089/6-1}) and the EPSRC CCP CoDiMa
(\textbf{EP/M022641/1}) for supporting this work. The first author would
like to thank the Royal Society, and the EPSRC (\textbf{EP/M003728/1}).
The second author would
like to acknowledge support from the OpenDreamKit Horizon 2020 European Research
Infrastructures Project (\#676541). The third author wishes to thank the
Computer Science Department of the University of St~Andrews for its hospitality
during numerous visits, and Karin Helbich for the pictures in this article.

\section{Notation, basic results and examples}
\label{sec:notation}

We mostly use standard notation for permutation groups and related objects and
refer the reader to references such as \cite{dixon1996permutation}.

Throughout we let $\Omega$ be a finite set and $n \in \N$.
We use $\Sym(\Omega)$ as notation for the \textbf{symmetric group on $\Omega$}
and $\s_n$ for the symmetric group on $\{1,...,n\}$.

\begin{definition}[Ordered Partitions]\label{basicdef}
\mbox{}
\begin{itemize}
\item
An \textbf{ordered partition} of $\Omega$ is an ordered list of disjoint
subsets (called \textbf{cells}) of $\Omega$ whose union is all of $\Omega$.
For all $i \in \Omega$ we write $\Delta_P(i)$ for the cell of $P$ that contains the point $i$, and we
let $\OPart(\Omega)$ denote the set of \textbf{all ordered partitions of
$\Omega$}.
The notation for ordered partitions will be explained in Example \ref{firstex}.
If $P\in \OPart(\Omega)$ and $i,j \in \Omega$, then we write $i \sim_P j$ if and only if 
$i \in \Delta_P(j)$.

\item
  We say that \textbf{$Q$ is finer than $P$} and write $Q \preccurlyeq P$ if and only if, for all $i,j \in \Omega$, it is true that
   $i \sim_Q j$ implies $i \sim_P j$. 
  Conversely, we say that $P$ is \textbf{coarser} than $Q$ in this situation.

  The relation $\preccurlyeq$ defines a partial order on $\OPart(\Omega)$ and we point out that, for all $P\in \OPart(\Omega)$,
  it is true that $P \preccurlyeq P$.

  We call an ordered partition \textbf{discrete} if and only if every element of $\Omega$ is in a
  cell by itself, and we call an ordered partition \textbf{trivial} if and only if $\Omega$ itself is its only cell.

\item
  If $P$ is an ordered partition of $\Omega$ and $g \in G$, then we write $P^g$ for the ordered
  partition that we obtain by applying $g$ to the elements in the cells of $P$.

\item
  Important ordered partitions come from the action of subgroups of $\Sym(\Omega)$ on
  $\Omega$. We call them \textbf{ordered orbit partitions}.
  If $H\le\Sym(\Omega)$ and $P\in\OPart(\Omega)$, then we say that $P$ is an ordered
  orbit partition for $H$ if and only if the cells of $P$ are exactly the orbits
  of $H$ on $\Omega$.
  We point out that, if $H$ has more than one orbit, then there are several
  distinct ordered orbit partitions, differing only by the ordering of the cells.

\item
  Suppose that $P$ and $Q$ are ordered orbit partitions of $\{1,...,n\}$, that $k \in \N$ and that
  $\Delta_1,...,\Delta_k$ are exactly the cells of $P$.
Then we write $\Sym(P)$ for the subgroup
  $\Sym(\Delta_1)\times\cdots\times\Sym(\Delta_k)$ of $\Sym(\Omega)$, i.e. the
 \textbf{stabilizer of the ordered partition $P$ in $\Sym(\Omega)$}. We use $\Co(P,Q)$ for the set
  of permutations in $\Sym(\Omega)$ that map \(P\) to \(Q\). Note that $\Co(P,Q)$
  will either empty or a coset of $\Sym(P)$. Moreover $\Co(P,P)=\Sym(P)$.

\end{itemize}

\end{definition}

\begin{ex}\label{firstex}
Given the algorithmic background of our work, we view partitions as lists.
For example $P:=[1,2,3,4 \mid 5,6,7]$ is an ordered partition of $\Omega:=\{1,2,3,4,5,6,7\}$ with cells $\{1,2,3,4\}$ and $\{5,6,7\}$. 
Then $P=[2,1,3,4 \mid 5,6,7]$, because the
  ordering of elements within a cell is irrelevant. But $P \neq [5,6,7 \mid 1,2,3,4]$, because the ordering
  of cells is relevant. However, we see that ${P \preccurlyeq
  [5,6,7 \mid 1,2,3,4]}$.
  Let $Q:=[1,2,3 \mid 4 \mid 5,6,7]$. Then $Q \preccurlyeq P$ and  $Q \preccurlyeq [4,5,6,7 \mid 3,2,1]$.

Next we consider the stabilizer in $\s_7$ of $P$.
  If $H_1 \le \s_7$ is the subgroup stabilizing the set $\{1,2,3,4\}$ and $H_2
  \le \s_7$ is the subgroup stabilizing the set $\{5,6,7\}$, then $\Sym(P)$ is
  exactly $H_1 \times H_2$.

  Finally, we look at the subgroup $H:=\langle (1\,2\,3), (5\,7)\rangle$ of $\s_7$ and
  we write down two of its ordered orbit partitions: $[1,2,3 \mid 4 \mid 5,7 \mid 6]$ and $[4 \mid 6 \mid 5,7 \mid 1,2,3]$.
\end{ex}

\begin{definition}[Meet]
Let $P, Q \in \OPart(\Omega)$.
Then we define 
the \textbf{meet of $P$ and $Q$}, denoted by $P \wedge Q$, as follows: 
$i,j \in \Omega$ are in the same cell of
 $P \wedge Q$ if and only if $i \sim_P j$ and $i \sim_Q j$.
For every cell of $P \wedge Q$ there is a unique pair of cells of $P$ and $Q$
that its elements are from, and the cells of $P \wedge Q$ are ordered
lexicographically with respect to these pairs.
\end{definition}

\begin{rem}
There are several reasons why ordered partitions are used in partition backtrack. One reason is that we can represent all elements of the group as coset representatives of some coset of the stabilizer of an ordered partition. This approach does not work with unordered partitions. Another reason is that, with ordered partitions, the relation ``meet'' is compatible with taking stabilizers of ordered partitions. This property no longer holds if unordered partitions are used.
One final comment:
The relation ``meet'' is not symmetric, as is illustrated in the following example.
\end{rem}

\begin{ex}
Let $P:=[1,2,3,4 \mid 5,6,7]$ and $Q:=[1,2 \mid 5,3 \mid 7,4,6]$ be ordered partitions of $\Omega:=\{1,2,3,4,5,6,7\}$.
We calculate $P \wedge Q$: For each element $i \in \Omega$, we find the indices of the cells $\Delta_P(i)$ and $\Delta_Q(i)$, in this order. This gives the following pairs:

  $1$ --- $(1,1)$,
  $2$ --- $(1,1)$,
  $3$ --- $(1,2)$,
  $4$ --- $(1,3)$,
  $5$ --- $(2,2)$,
  $6$ --- $(2,3)$,
  $7$ --- $(2,3)$.

Therefore $P \wedge Q = [1,2 \mid 3 \mid 4 \mid 5 \mid 6,7]$.
Calculating $Q \wedge P$ instead gives the ordered partition $[1,2 \mid 3 \mid 5 \mid 4 \mid 6,7]$; it has the same cells as $P \wedge Q$, but in a different order.
\end{ex}

It will be important later that taking meets of ordered partitions of $\Omega$ is
compatible with the action of elements of $\Sym(\Omega)$ on $\Omega$.

\begin{lem}\label{lem:meetandact}
  Let $H\leq\Sym(\Omega)$, let $P$ and $Q$ be ordered partitions of $\Omega$, and let
  $h \in H$.

  Then ${(P \wedge Q)}^h = P^h \wedge Q^h$.
\end{lem}

\begin{proof}
First we note: If $g \in \Sym(\Omega)$, then $i \sim_{P^g}j$ if and only if
  $i^{g^{-1}} \sim_P j^{g^{-1}}$. Let $T:=P \wedge Q$. Then by definition
  $i \sim_T j$ if and only if $i \sim_P j$ and $i \sim_Q j$.
 Now, for all $h \in H$:
  $$ i\sim_{T^h} j \Leftrightarrow i^{h^{-1}} \sim_T j^{h^{-1}}
                   \Leftrightarrow i^{h^{-1}} \sim_P j^{h^{-1}} \mbox{ and } i^{h^{-1}} \sim_Q j^{h^{-1}} \Leftrightarrow i \sim_{P^h} j \mbox{ and } i\sim_{Q^h} j
                   \Leftrightarrow i \sim_{P^h \wedge Q^h} j.
  $$
So far we proved that ${(P \wedge Q)}^h \preccurlyeq P^h \wedge
Q^h$
and 
$P^h \wedge Q^h \preccurlyeq {(P \wedge Q)}^h$, but
this does not imply equality.
We can describe the cell  $\Delta_{P \wedge Q}(i)$ by the pair $[\Delta_P(i), \Delta_Q(i)]$. So we argue as follows:
$$
 \Delta_{{(P \wedge Q)}^h}(i) = \Delta_{(P \wedge Q)}(i^{h^{-1}}) = [\Delta_P(i^{h^{-1}}),\Delta_Q(i^{h^{-1}})] = [\Delta_{P^h}(i), \Delta_{Q^h}(i)] = \Delta_{P^h\wedge Q^h}(i).
$$  
  
This implies that ${(P \wedge Q)}^h = P^h\wedge Q^h$.
\end{proof}

\section{Refining Ordered Partitions}
\label{sec:refiner}

\subsection{Partition Backtrack}

Partition backtrack is a technique for solving search problems on permutation
groups. It is implemented in \cite{GAP4} and Magma (\cite{MR1484478}) as the main technique for solving a range of problems, including
computing group and coset intersections, set and partition stabilizers,
centralizers, and normalizers. This paper will not provide a full description of
partition backtrack, instead we refer readers to~\cite{Leon}.

In brief, partition backtrack searches for elements of $\Sym(\Omega)$ that
satisfy a list of properties. These properties will typically be of the form
``element lies in a subgroup of $\Sym(\Omega)$'' or ``element lies in a coset of
a subgroup of $\Sym(\Omega)$''. For example, finding the stabilizer of a set $S$ in a subgroup $H$ of $\Sym(\Omega)$ can
be expressed as finding all elements that satisfy the list of properties ``element stabilizes $S$ in $\Sym(\Omega)$'' and
``element is contained in the subgroup $H$''.
We will not specify ``properties'' further because we have many different applications in mind, but each property should be easy to verify with an algorithm. 
The groups that appear in the search can be expressed in a variety of
ways, for example by a set of generators, as the normalizer of a
group, as the automorphism group of a graph, or as the stabilizer of a set or
ordered partition in $\Sym(\Omega)$.

To give a more precise description: Partition backtrack takes a list of properties as input, then it starts from the pair $(P_0,Q_0)$ of trivial
ordered partitions and proceeds to perform search by repeatedly alternating between
the following two phases, starting at $i=0$:

\begin{enumerate}

\item \textbf{Refinement phase:} Start with the first property of the list
and refine the pair of ordered partitions $(P_i,Q_i)$ to a new pair $(P,Q)$ such that the
following holds: Every element of $\Sym(\Omega)$ in $\Co(P_i,Q_i)$ is also in
$\Co(P,Q)$. Hence, we have not removed any permutations that satisfy all
properties on the list.

Repeat this process with the second property from the list and continue through all
properties. Start again at the beginning of the list until the ordered partition
does not change anymore. The resulting pair of ordered partitions will be called
$(P_{i+1},Q_{i+1})$.

\item \textbf{Branching phase:} Take $P_{i+1}$ and $Q_{i+1}$. At this point,
three cases are possible:

\begin{enumerate}

\item There is no permutation in $\Sym(\Omega)$ that maps $P_{i+1}$
to $Q_{i+1}$. This means that this part of the search does not produce
any permutation that satisfies all properties on the list. 

\item Every cell in $P_{i+1}$ and $Q_{i+1}$ is of size one. Then there
is exactly one permutation that maps $P_{i+1}$ to $Q_{i+1}$. Perform a final
check that this satisfies all properties on the list, and if it does, then record it as a
solution.

\item Split the search by producing a list of pairs of ordered partitions
$(P'_1,Q'_1),\dots,(P'_k,Q'_k)$ where the $\Co(P'_j,Q'_j)$ for $j \in
\{1,\ldots,k\}$ form a disjoint union of $\Co(P_{i+1},Q_{i+1})$.

The practical method that we choose for splitting is the following:

We choose $l \in \N$ such that the $l$-th cell $c$ of $P_{i+1}$ has size at least $2$, and we choose
a single element $a$ from $c$. We make all the $P'_j$ equal to $P_{i+1}$, except that
$a$ is removed from the cell $c$ and placed in a new cell at the end, by itself.
For each element $b$ of the $l$-th cell $d$ of $Q_{i+1}$, we create $Q'_j$, which
is identical to $Q_{i+1}$ except that $b$ is removed from cell $d$ and placed in a
new cell at the end, by itself.

This gives new pairs of ordered partitions that are finer than $(P_{i+1}, Q_{i+1})$.
For each of these pairs, the search continues by going to the refinement phase.
\end{enumerate}

\end{enumerate}

This process will stop eventually because 
there is a finite number of branches, they all have finite length, and the refinement stops if the partition does not change anymore. In~\cite{Leon} the
author explains how this algorithm can be implemented efficiently.

This paper will focus on refinement, because the quality of refiners
is one of the main influences on performance.
Refiners are only defined a single ordered partition -- Lemma \ref{lem:backtrack} explains how refiners are used on cosets.

\begin{definition}[M-refiner]\label{def:refine}
Let $M \subseteq \Sym(\Omega)$ be a set of elements with a given property.
An \textbf{$M$-refiner} is a map
\(f_M:\OPart(\Omega) \rightarrow \OPart(\Omega) \) that satisfies the following
conditions for any ordered partition \(P\) of \(\Omega\):
  \begin{itemize}
  \item[(a)] $f(P) \preccurlyeq P$.
  \item[(b)] For all $g \in M$, it is true that $f(P^g)
    = {f(P)}^g$.
  \end{itemize}
\end{definition}

Part (b) of the definition implies that every element in $\Sym(P)$ that satisfies the
property (referring to $M$) and stabilizes $P$ also stabilizes $f(P)$.
This turns out to be a very natural condition -- our
experience is that refiners tend to satisfy (b). It is not only a natural requirement for a refiner, but it is
also one of the main reasons why partition backtrack is so efficient.
For more details see Sections 6 and 7 in~\cite{Leon}.

Lemma~\ref{lem:backtrack}
shows how, using an $M$-refiner (which refers to only a single ordered
partition), we can perform filtering on cosets, as required by our search
defined above. This greatly simplifies our implementation.

\begin{lem}\label{lem:backtrack}
Let $P,Q \in \OPart(\Omega)$ and let
$f_M$ be an $M$-refiner for $M \subseteq Sym(\Omega)$. 
Then for all $g \in M$ such that
$g \in \Co(P,Q)$, it follows that $g \in \Co(f_M(P), f_M(Q))$.
\end{lem}

\begin{proof}
If \(g \in \Co(P,Q)\), then \(Q = P^g\), and by definition of an $M$-refiner it holds that $f_M(Q) = f_M(P^g) = {f_M(P)}^g$, and therefore
\(\Co(f_M(P),f_M(Q)) = \Co(f_M(P), {f_M(P)}^g)\), and so

\(g \in \Co(f_M(P),f_M(Q))\).
\end{proof}

\subsection{Refiners for Permutation Groups}

We will now give a definition of the standard refiner for permutation groups given by a list of generators from~\cite{Leon}. We will introduce refiners that use orbital graphs in Section~\ref{sec:orbrefiner}.

\begin{definition}[$\Fixed_M$]\label{def:genrefine}
  Let $M$ be a subgroup of $\Sym(\Omega)$. Then the map $\Fixed_M : \OPart(\Omega) \rightarrow \OPart(\Omega)$
  is defined as follows:

  \begin{itemize}
  \item Given $P \in \OPart(\Omega)$, let $k \in \N$ and
    $\alpha_1,\ldots,\alpha_k \in \Omega$ be the elements in singleton cells of $P$.
    Let \(M_0\) denote the point-wise stabilizer of $\alpha_1,\ldots,\alpha_k$ in $M$.
  \item Return the meet of \(P\) and an ordered orbit partition of $M_0$.
  \end{itemize}
\end{definition}

This map does not necessarily give a refiner as it is: The reason is that we do
not define the ordering of the cells in the resulting ordered orbit partition.
In the implementation this is fixed by producing a list of orbits,
outputting the cells in arbitrary order and then fixing this ordering for
later instances.
The details are given in the proof of the next lemma.

\begin{lem}\label{Fixedref}
  Let $M$ be a subgroup of $\Sym({\Omega})$. Then the map $\Fixed_M$ is a refiner.
\end{lem}

\begin{proof}
  Let $P \in\OPart(\Omega)$. Then $\Fixed_M(P) \preccurlyeq P$,
  because $\Fixed_M(P)$ is the meet of $P$ with another ordered partition and it is
  therefore finer than $P$.

  Let $k \in \N$ and $\alpha_1,\ldots,\alpha_k \in \Omega$ be such that these are
  precisely the elements in singleton cells of $P$. Now we let \(M_0\) denote the point-wise stabilizer of $\alpha_1,\ldots,\alpha_k$ in $M$ and we note that $F:=\la M_0 \ra$
  stabilizes $\alpha_1,\ldots,\alpha_k$.
  Let $g \in G$. We need to show that $\Fixed_M(P^g) = {\Fixed_M(P)}^g$.
  First we note that $F^g$ fixes $\alpha_1^g,\ldots,\alpha_k^g$, which are
  exactly the singleton cells of $P^g$.
  Now we fix some ordered orbit partition $Q$ of $F$ as described before the lemma. Then $Q^g$ is an ordered orbit partition of $F^g$. Using Lemma~\ref{lem:meetandact} we deduce:
  \begin{align*} {\Fixed_M(P)}^g &= {(P\wedge Q)}^g
                                 = P^g\wedge{Q}^g
                                 = \Fixed_M(P^g).
  \end{align*}
\end{proof}

The major limitation of $\Fixed$ is that it ignores non-singleton cells. More
concretely, given a transitive group \(G\) and an ordered partition \(P\) that
contains no singleton cells, \(\Fixed_G(P) = P\). We cannot easily use the same
strategy as \(\Fixed\) for non-singleton cells (finding the stabilizer of the
non-singleton cells in \(G\)), because this would require solving the set
stabilizer problem, which is exactly one of the problems that is solved via backtrack!

Instead, we look at other properties of groups we can use, which allow us to
refine non-singleton cells efficiently. We will now show how orbital graphs can
be used in refiners that complement existing refiners for groups
expressed by a list of generators. These refiners will provide useful refinement
even for transitive groups and non-singleton cells.
We are not the first ones to use this idea:
\cite{Th} uses orbital graphs as an ingredient for refiners for normalizer search.
However, our work does not build on his -- partly because our hypothesis is more general, and partly because his results have not been published except for in his PhD thesis. We show
that the concept has not yet been fully exploited.

\section{Orbital graphs}\label{sec:orbrefiner}

Here we introduce the graphs that we use in our new refiners -- orbital graphs -- and prove the properties that are necessary in order to decide whether or not a refinement by orbital graphs is computationally beneficial.

\begin{definition}[Digraphs]
For the purposes of this paper, a \textbf{digraph $\Gamma$} is a pair
$\Gamma=(V,A)$ where $V$ denotes the set of \textbf{vertices (or points)} and
$A$ denotes the set of \textbf{arcs}, i.e. directed edges. If $x,y \in V$, then
an arc from $x$ to $y$ in $\Gamma$ will be denoted by $(x,y)$.
An \textbf{isolated vertex} of a digraph is a vertex with no arcs going into it or coming out of it.
A digraph is \textbf{complete} if and only
if its set of arcs is exactly $\{(\omega_1,\omega_2) \mid \omega_1,\omega_2 \in \Omega,
\omega_1 \neq \omega_2\}$.

$\Gamma$ is a \textbf{complete bipartite digraph} if and only if there
exist disjoint subsets $S,E$ of vertices such that $V$ is the union of $S$ (the ``starting'' vertices) and $E$ (the ``end'' vertices) and the
set of arcs is exactly $A = \{(\omega_1,\omega_2) \mid \omega_1 \in S, \omega_2\in
E\}$.
\end{definition}

We refer the reader to \cite{Bang-Jensen:2008:DTA:1523254} for standard notation
and for the definitions of connected components, graph isomorphisms etc.
We point out that by a \textbf{proper digraph} we mean a digraph that has at least one arc such that its reverse arc is not in the graph.
All digraphs considered here have no multiple arcs and no loops. We also point out that whenever we refer to the \textbf{size of a connected component} we mean the number of vertices in the component.

\begin{definition}[Orbital Graphs]
Let  $\Omega$ be a finite set, $G:=\Sym(\Omega)$ and $H \le G$. 
For all vertices $\gamma \in \Gamma$ and all $h \in H$ we write $\gamma^h$ for the image of $\gamma$ under $h$ in the original permutation action.

Now let
$\alpha,\beta \in \Omega$ be distinct elements, chosen in this order. We define a
digraph $\Gamma=(\Omega,A)$ where the set of arcs $A$ is 
defined as $A:=\{(\alpha^h,\beta^h) \mid h \in H\}$.
This digraph is called the \textbf{orbital graph of $H$ with base-pair $(\alpha,\beta)$}, and is denoted by $\Gamma(H, \Omega, (\alpha, \beta))$.

Following \cite{dixon1996permutation} we say that an orbital graph is \textbf{self-paired} if and only if, for all $\gamma,\delta \in \Omega$, it is true that $(\gamma,\delta)$ is an arc if and only if $(\delta,\gamma)$ is an arc.
\end{definition}

\begin{ex}\label{orbex}
If we build an orbital graph for $\s_3$, then for each base-pair we obtain the
complete digraph on $\{1,2,3\}$. The reason is that this group is
$2$-transitive: Given $\alpha,\beta \in \{1,2,3\}$ such that $(\alpha,\beta)$ is a base-pair, and
given any distinct $\gamma,\delta \in \{1,2,3\}$, there exists some $g \in \s_3$ such that $\alpha^g=\gamma$
and $\beta^g=\delta$. Therefore $(\gamma,\delta)$ is also an arc, and this means that all possible
arcs exist.
For groups that are not $2$-transitive, the choice of the base-pair becomes much
more important. Let $H:=\langle (1\,2\,3),(4\,5),(4\,6)\rangle \le \s_6$.

Then, starting with the base-pair $(1,2)$, we obtain the following digraph:

\begin{center}
\begin{tikzpicture}[shorten >=1pt,->]
  \tikzstyle{vertex}=[circle,fill=black!25,minimum size=14pt,inner sep=0pt]

  \foreach \name/\angle/\text in {P-1/210/1, P-2/90/2, 
                                  P-3/-30/3}
  \node[vertex,xshift=-1.5cm] (\name) at (\angle:1cm) {$\text$};

  \foreach \name/\angle/\text in {Q-1/210/4, Q-2/90/5, 
                                  Q-3/-30/6}
  \node[vertex,xshift=1.5cm] (\name) at (\angle:1cm) {$\text$};

  \foreach \from/\to in {1/2,2/3,3/1}
    { \draw (P-\from) -- (P-\to); }
  \draw[very thick] (P-1) -- (P-2);
\end{tikzpicture}
\end{center}

But, starting with the base-pair $(5,6)$, we find:

\begin{center}
\begin{tikzpicture}[shorten >=1pt,->]
  \tikzstyle{vertex}=[circle,fill=black!25,minimum size=14pt,inner sep=0pt]

  \foreach \name/\angle/\text in {P-1/210/1, P-2/90/2, 
                                  P-3/-30/3}
  \node[vertex,xshift=-1.5cm] (\name) at (\angle:1cm) {$\text$};

  \foreach \name/\angle/\text in {Q-1/210/4, Q-2/90/5, 
                                  Q-3/-30/6}
  \node[vertex,xshift=1.5cm] (\name) at (\angle:1cm) {$\text$};

%  \foreach \from/\to in {1/2,2/3,3/1,2/1,3/2,1/3}
%    { \draw (Q-\from.225) -- (Q-\to.75); }

  \draw (Q-1.45) -- (Q-2.255);
  \draw (Q-1.345) -- (Q-3.195);
  \draw (Q-2.225) -- (Q-1.75);
  \draw[very thick] (Q-2.285) -- (Q-3.135);
  \draw (Q-3.105) -- (Q-2.315);
  \draw (Q-3.165) -- (Q-1.15);
\end{tikzpicture}
\end{center}

For a connected digraph we start with the pair
$(2,4)$:

\begin{center}
\begin{tikzpicture}[align=center,shorten >=1pt,->]
  \tikzstyle{vertex}=[circle,fill=black!25,minimum size=14pt,inner sep=0pt]

  \foreach \name/\x/\y in {1/1/2.5, 2/2.5/2.5, 3/4/2.5, 4/1/1, 5/2.5/1, 6/4/1}
  \node[vertex] (\name) at (\x,\y) {$\name$};

  \foreach \from/\to in {1/4, 1/5, 1/6, 2/4, 2/5, 2/6, 3/4, 3/5, 3/6}
    { \draw (\from) -- (\to); }
  \draw[very thick] (2) -- (4);
\end{tikzpicture}
\end{center}

\end{ex}

Some properties of orbital graphs can be found in \cite{C} and \cite{dixon1996permutation}, but we decided to include short proofs for the statements in the next lemma in order to make this article more self-contained.

\begin{hyp}\label{orbhyp}
Let $\Omega$ be a finite set, let $H \le G:=\Sym(\Omega)$ and let $\alpha,\beta \in \Omega$ be distinct. Let
$\Gamma:=\Gamma(H, \Omega,(\alpha,\beta))$ and let $A$ denote the set of arcs of $\Gamma$.
\end{hyp}

\begin{lem}\label{basic}
Suppose that Hypothesis \ref{orbhyp} holds.
Then we have the following:

\begin{enumerate}
\item[(i)]
$\Gamma=\Gamma(H,\Omega,(\gamma,\delta))$ if and only if $(\gamma,\delta) \in A$.

\item[(ii)]
$\Gamma$ is self-paired if and only if some $h \in H$ interchanges $\alpha$ and $\beta$.

\item[(iii)]
$\alpha^H$ is precisely the set of vertices of $\Gamma$ that are the starting point of some arc.

\item[(iv)]
$\beta^H$ is precisely the set of vertices of $\Gamma$ that are the end point of some arc.

\item[(v)]
The number of arcs starting at $\alpha$ is $|\beta^{H_\alpha}|$ and
the number of arcs going into $\beta$ is
$|\alpha^{H_\beta}|$.

\end{enumerate}
\end{lem}

\begin{proof}
(i) If $\Gamma=\Gamma(H,\Omega,(\gamma,\delta))$, then by definition $(\gamma,\delta)$ is an arc in $\Gamma$.

Conversely, suppose that $(\gamma,\delta)$ is an arc in $\Gamma$. Then there exists some $h \in H$ such that
$(\alpha^h,\beta^h)=(\gamma, \delta)$. Hence the orbital graph with base-pair $(\alpha,\beta)$ is the same as the orbital graph with base-pair $(\gamma, \delta)$.

(ii)
By (i) $\Gamma$ coincides with $\Gamma(H,\Omega, (\beta, \alpha))$ if and only if the arc $(\beta,\alpha)$ exists in $\Gamma$, which happens if and only if
there exists some $h \in H$ such that $\alpha^h=\beta$ and $\beta^h=\alpha$.

(iii)
Let $\gamma \in \alpha^H$ and $h \in H$ be such that $\gamma=\alpha^h$.
Then $(\gamma,\beta^h)$ is an arc with starting point $\gamma$.
Conversely, if $\delta \in \Omega$ is such that $(\gamma,\delta)$ is an arc in $\Gamma$, then there exists some $h \in H$ such that $(\alpha^h,\beta^h)=(\gamma,\delta)$ and hence $\gamma=\alpha^h \in \alpha^H$.
Similar arguments show (iv).

(v) The number of arcs starting at $\alpha$ is

$|\{(\alpha,\gamma) \mid \gamma \in \Omega\}|=|\{(\alpha^h,\beta^h) \mid h \in H, \alpha^h=\alpha\}|=|\beta^{H_\alpha}|$ and the number of arcs going into $\beta$ is

$|\{(\delta,\beta) \mid \delta \in \Omega\}|=|\{(\alpha^h,\beta^h) \mid h \in H, \beta^h=\beta\}|=|\alpha^{H_\beta}|.$\\
\end{proof}

\begin{rem}\label{rem} Some comments:

(a) Parts (iii) and (v) of the lemma, together, give the total number of arcs in $\Gamma$.
The number of arcs starting at $\alpha$ is exactly $|\beta^{H_\alpha}|$, so we obtain
 $|A|=|\alpha^H| \cdot |\beta^{H_\alpha}|$.

(b)
In (ii) it is not true that $H$ must contain the transposition $(\alpha, \beta)$.
A counterexample is provided by
$H:=\la (1\,2)(3\,4)\ra \le \s_4$ acting naturally on $\Omega:=\{1,2,3,4\}$ and its orbital graph
with base-pair $(1,2)$.

(c)
Parts (iii) and (iv) of the lemma imply that, if $H$ acts transitively on $\Omega$, then $\Gamma$ has no isolated vertices.
\end{rem}

\begin{lem}\label{act}
Suppose that Hypothesis \ref{orbhyp} holds. Then $H$ acts on $\Gamma$ as a group of graph automorphisms.
\end{lem}

\begin{proof}
First we note that $H$ acts faithfully on the set $\Omega$.
Now we let $\gamma,\delta \in \Omega$.

If $(\gamma,\delta)$ is an arc, then there exists some $h \in H$ such that $(\gamma,\delta)=(\alpha^h,\beta^h)$ by definition of $\Gamma$.
Hence $(\gamma^g,\delta^g)=(\alpha^{hg},\beta^{hg})$ is an arc.
Conversely, if $(\gamma^h,\delta^h)$ is an arc, then there exists some $a \in H$ such that $(\gamma^h,\delta^h)=(\alpha^a,\beta^a)$ and hence $(\gamma,\delta)=(\alpha^{ah^{-1}},\beta^{ah^{-1}})$ is an arc.
As $H$ is a group, the induced maps are bijective and hence
every $h \in H$ induces a graph automorphism on $\Gamma$.
\end{proof}

\begin{lem}\label{transimp}
Suppose that Hypothesis \ref{orbhyp} holds and let \(\Delta\) denote the connected component that contains \((\alpha, \beta)\).
Then every connected component of \(\Gamma\) that has size at least $2$ is isomorphic to \(\Delta\).
\end{lem}

\begin{proof}
Let \({\Delta'}\) denote an arbitrary connected component of \(\Gamma\) of size at least $2$ and let \((\gamma,\delta)\) be an arc in $\Delta'$.

From the definition of orbital graphs let \(h\in H\) be such that \((\alpha^h,\beta^h)=(\gamma,\delta)\). Then \(h\) induces an automorphism on \(\Gamma\) by Lemma \ref{act} and it moves all arcs from \(\Delta\) to arcs in \(\Delta'\).
Conversely, \(h^{-1}\) induces an automorphism on \(\Gamma\) that moves all arcs of \(\Delta'\) into \(\Delta\). Thus it follows that \(\Delta\) and \(\Delta'\) are isomorphic as graphs.
\end{proof}

Lemma \ref{lem:orbitalorbits} shows how to choose a set of base-pairs that determines all orbital graphs for a group \(H\). Parts (i) and (ii) show to take a representative from each orbit of \(H\) as the first element of the base-pair, and then part (iii) shows that we must stabilize this representative in \(H\), and take a representative from each orbit in this stabilizer for the second element of our base-pair. These base-pairs will allow us to analyze the set of orbital graphs of a group, before we construct any orbital graphs explicitly.

\begin{lem}\label{lem:orbitalorbits}
Let $\Omega$ be a finite set and let $H \le G:=\Sym(\Omega)$.

\begin{enumerate}[(i)]
\item  Suppose that $\alpha,\beta \in \Omega$ and $\alpha \in \beta^H$. Then the set of orbital graphs
  of $H$ with base-pairs starting with $\alpha$ is equal to the set of orbital
  graphs of $H$ with base-pairs starting with $\beta$.

\item Suppose that $\alpha,\gamma \in \Omega$ and $\alpha \notin \gamma^H$. Then the set of orbital
  graphs of $H$ with base-pairs starting with $\alpha$ is disjoint from the set of
  orbital graphs of $H$ with base-pairs starting with $\gamma$.

\item Suppose that $\alpha,\beta,\gamma \in \Omega$ and that $\alpha \neq \beta$, $\alpha \neq \gamma$. Let $\Gamma_1:=\Gamma(H,\Omega,(\alpha,\beta))$ and $\Gamma_2:=\Gamma(H,\Omega,(\alpha,\gamma))$.
Then $\Gamma_1=\Gamma_2$ if and only if
 $\gamma \in \beta^{H_\alpha}$.
\end{enumerate}
\end{lem}

\begin{proof}
\begin{enumerate}[(i)]
\item Let $h \in H$ be such that $\alpha^h=\beta$. Then for all $\gamma \in \Omega$,
  it follows that $\Gamma(H,\Omega,(\alpha,\gamma))=\Gamma(H,\Omega,(\alpha^h,\gamma^h)) = \Gamma(H,\Omega,(\beta,\gamma^h))$. Conversely $\Gamma(H,\Omega,(\beta,\gamma))=\Gamma(H,\Omega,(\beta^{(h^{-1})},
  \gamma^{(h^{-1})})) = \Gamma(H,\Omega,(\alpha,\gamma^{(h^{-1})})$.

\item Suppose that the pairs $(\alpha,\beta)$ and $(\gamma,\delta)$ generate the same orbital graph
  of $H$. Then by Lemma \ref{basic}~(i) there is $h \in H$ such that $(\alpha^h,\beta^h) = (\gamma,\delta)$, which
  implies that $\alpha \in \gamma^H$. This proves the statement.

\item If $\Gamma_1=\Gamma_2$, then $(\alpha,\beta)$ and $(\alpha,\gamma)$ generate the same orbital graph. So by Lemma \ref{basic}~(i) there is $h \in H$ such that $(\alpha^h,\beta^h)=(\alpha,\gamma)$. This means that $\alpha^h=\alpha$ and therefore
 $h \in H_\alpha$, which implies that $\gamma \in \beta^{H_\alpha}$. Conversely, if $\gamma \in \beta^{H_\alpha}$ then there
 exists $h \in H_\alpha$ such that $(\alpha^h,\beta^h)=(\alpha,\gamma)$ and hence $\Gamma_1=\Gamma_2$.
\end{enumerate}
\end{proof}

\subsection{Futile orbital graphs}

After the preparatory results above, we now characterize the situations where orbital graphs are beneficial in partition backtrack.

\begin{definition}[Futile orbital graph]
Suppose that Hypothesis \ref{orbhyp} holds and that $P$ is an ordered orbit
partition of $H$. We denote the stabilizer of $P$ in $G$ by $\Sym(P)$, as we did in Definition \ref{basicdef}, and we
emphasize that $\Sym(P)$ stabilizes every $H$-orbit (i.e. every cell of the ordered
partition $P$) as a set and that it acts as the full symmetric group on every
orbit.

We say that the orbital graph $\Gamma$ is \textbf{futile} if and only if $\Sym(P)$,
in its natural action on $\Omega$, induces graph automorphisms on
$\Gamma$.
\end{definition}

\begin{ex}
We refer to the graphs in Example \ref{orbex} for the group $H:=\langle (1\,2\,3),(4\,5),(4\,6)\rangle \le \s_6$, and we use the ordered orbit partition $P:=[1,2,3|4,5,6]$.
The first graph $\Gamma_1$, with base-pair $(1,2)$, is not futile. We observe
that the transposition $(1\,2) \in \s_6$ stabilizes the partition $P$, but it does not induce a graph automorphism on $\Gamma_1$ because $(1,2)$ is an arc in $\Gamma_1$ and $(2,1)$ is not. 

The second graph $\Gamma_2$, with base-pair $(5,6)$, is futile. 

For this we note that $\Sym(P)=\langle (1\,2),(2\,3),(4\,5),(4\,6)\rangle$.
Now if $g \in \Sym(P)$, then $g$ permutes the three isolated vertices in the graph and it permutes the set of vertices in the connected component containing $4,5$ and $6$.
Given that this connected component is a complete graph, it follows that $g$ induces a graph automorphism on this component and hence on all of $G$. 

Finally we look at the third graph $\Gamma_3$ with base-pair $(2,4)$.
This is a complete bipartite graph, so again we see that $\Sym(P)$ induces graph automorphisms.
\end{ex}

The intuition behind this definition is that we would like to be able to characterize orbital graphs where the graph structure does not give us any additional information compared to the orbit structure on $\Omega$ that comes from the action of $H$. 
In a futile orbital graph, all the information that could be gained from the graph structure can already be seen in an ordered orbit partition of $\Omega$ with respect to $H$. Building such a graph would be useless from a computational perspective.
Therefore our main theoretical result on this topic classifies futile orbital graphs. We also discuss how to detect futile orbital graphs without even building them. 

We note that the following
result does not place any restrictions about the number of orbits of $H$ on
$\Omega$. In particular there could be arbitrarily many isolated points in
$\Gamma$. 

\begin{thm}\label{use}
Suppose that Hypothesis \ref{orbhyp} holds.
Then $\Gamma$ is futile if and only if it has a unique connected component
$\Delta$ of size at least $2$ and moreover one of the following holds:

(a) $\Delta$ is a complete bipartite digraph or 

(b) $\Delta$ is a complete digraph.
\end{thm}

\begin{proof}
Let $P$ be an ordered orbit partition of $H$. Then $\Sym(P)$ acts on the set of orbits of
$H$ and it acts faithfully on the set of vertices of $\Gamma$. Hence to answer the question
whether $\Gamma$ is futile or not, we only have to consider arcs in $\Gamma$.
 
Throughout the proof let $\Delta$ denote a connected component of size at least $2$, and without loss suppose that the arc $(\alpha,\beta)$ is contained in $\Delta$.

We split our proof into two cases depending on whether or not $\Gamma$ is a
proper digraph. In both cases we begin by proving that the futility of $\Gamma$ implies that $\Delta$ is the unique connected component of size at least $2$ and that (a) or (b) hold, and then we discuss the converse.\\

\textbf{Case 1:} $\Gamma$ is a proper digraph.

Then $\Gamma$ is not self-paired and Lemma \ref{basic}~(i) and (ii) imply that,
for all $\omega_1, \omega_2 \in \Omega$, there is at most one arc between them.
In the following arguments we will often refer to Lemma \ref{basic}~(iii) and
(iv) as well.

We suppose that $\Gamma$ is futile and we prove in a series of little steps that $\Delta$ is the unique connected component of size at least $2$ and that (a) is true.

\textbf{(1)} Suppose that $\gamma,\delta \in \Omega$ are distinct and in the
same $H$-orbit. Then they are not on an arc. In particular $\alpha^H \neq
\beta^H$.
\smallskip

\begin{roof}
As $\gamma$ and $\delta$ are in the same $H$-orbit, they lie in the same cell of
the partition $P$. It follows from the futility of $\Gamma$ that
the transposition $(\gamma,\delta) \in \Sym(\Omega)$, which stabilizes $P$,
induces a graph automorphism on $\Gamma$. Therefore neither $(\gamma,\delta)$
nor $(\delta,\gamma)$ is an arc. From this and the fact that $(\alpha, \beta)
\in A$ it follows that $\alpha^H \neq \beta^H$.
\end{roof}

\smallskip
\textbf{(2)} Suppose that $\omega \in \Omega$ is on an arc. Then it is either a
starting point or an end point, but not both.

\begin{roof}
This follows from Lemma \ref{basic}~(iii) and (1).
\end{roof}

Let $S:=\alpha^H$ and $E:=\beta^H$, and let $I \subseteq \Omega$ denote the set of isolated vertices of $\Gamma$.

\smallskip
\textbf{(3)} $\Omega=S \dot \cup E \dot \cup I$. Moreover $S \cup E$ spans $\Delta$, and $\Delta$
is a complete bipartite digraph.

\begin{roof}
The first statement follows from (2). Moreover there are no arcs between
vertices in $S$ or $E$, respectively, by (1). We show that all elements of $E$
are on an arc with $\alpha$:

For all $\gamma \in E$, we find the transposition $g:=(\beta,\gamma) \in \Sym(P)$,
and it fixes $\alpha^H$ point-wise by (1). The futility of $\Gamma$ implies that
$g$ maps the arc $(\alpha,\beta)$ to the arc $(\alpha,\gamma)$. Now it follows
that $A=S \times E$ and hence the digraph spanned by $S \cup E$ is a complete
bipartite digraph. Then it must coincide with $\Delta$ and $\Delta$ is the unique connected component of size at least $2$
of $\Gamma$.
\end{roof}

Conversely, we suppose that $\Delta$ is the unique connected component of size at
least $2$ of $\Gamma$ and that (a) holds. We prove that
$\Gamma$ is futile.

Let $S$ and $E$ denote the subsets of the vertex set of $\Gamma$ such that all
arcs start at $S$ and end at $E$. Let $I$ be the set of isolated vertices of
$\Gamma$, so that $\Omega=S \dot \cup E \dot \cup I$.

Now $\alpha^H \subseteq S$ and the bipartite structure implies that even $\alpha^H=S$.
Similarly $\beta^H =E$. Therefore $\Sym(P)$ stabilizes the sets $S$, $E$ and $I$. We
already know that $\Sym(P)$ permutes the vertices of $\Gamma$ faithfully, so now we
look at arcs.

Let $g \in \Sym(P)$ and let $(\omega_1,\omega_2) \in A$. Then $\omega_1 \in S$,
$\omega_2 \in E$ and there exists some $h \in H$ such that
$(\alpha^h,\beta^h)=(\omega_1,\omega_2)$. Since $\Sym(P)$ stabilizes the sets $S$
and $E$, we see that $\omega_1^g \in S$ and $\omega_2^g \in E$. The completeness
property then implies that $(\omega_1^g,\omega_2^g) \in A$.

Conversely, if $(\omega_1^g,\omega_2^g) \in A$, then there exists some $h \in H$
such that $(\alpha^h,\beta^h)=(\omega_1^g,\omega_2^g)$. Now
$\omega_1=\alpha^{hg^{-1}} \in S$ and $\omega_2=\beta^{hg^{-1}} \in E$ whence
$(\omega_1,\omega_2) \in A$ by completeness.

Hence $\Gamma$ is futile.\\

\textbf{Case 2:} $\Gamma$ is not a proper digraph, which means that it is self-paired.

We still have our connected component $\Delta$ and we begin, once more, with the hypothesis that $\Gamma$ is futile. Let $\gamma \in \Omega$ be an arbitrary, non-isolated
vertex.

We know that $\beta^H=\alpha^H$ by Lemma \ref{basic}~(iii) and (iv), because
$\Gamma$ is self-paired. As $\gamma$ was chosen to be a non-isolated vertex, there is some arc that starts or ends in $\gamma$. Therefore $\gamma \in \alpha^H$ and hence $\alpha,\beta,\gamma$ are all in the same
$H$-orbit and hence in a common cell of the partition $P$.
In particular the transposition $g:=(\beta,\gamma)$ is contained in $\Sym(P)$ and,
because of futility, it induces a graph automorphism on $\Gamma$.

Then $(\alpha,\beta) \in A$ implies that $(\alpha,\gamma)=(\alpha^g,\beta^g) \in
A$. This argument shows that $\Delta$ is the
only connected component of size at least $2$ in $\Gamma$ and that (b) is true.\\

We conversely suppose that $\Delta$ is the unique connected component of size at least $2$ of $\Gamma$ and that (b) holds. Together with the definition of orbital
graphs (and the fact that arcs always go both ways in the present case) this
implies that $\alpha^H=\beta^H$ spans $\Delta$ and that the isolated vertices, viewed as elements of $\Omega$, are
not contained in $\alpha^H$.

We know that $\Sym(P)$ acts faithfully on the vertex set of $\Gamma$. Now let $g \in
\Sym(P)$ and let $\omega_1,\omega_2 \in \Omega$. We recall that $\alpha^H=\beta^H$
is $\Sym(P)$-invariant.

Then it follows as in Case 1, using the completeness, that $(\omega_1,\omega_2)
\in A$ if and only if $(\omega_1^g,\omega_2^g) \in A$. Consequently $\Sym(P)$ acts as a group of graph
automorphisms on $\Gamma$, i.e. $\Gamma$ is futile.
\end{proof}

We give an example in order to illustrate that futility of an orbital graph is not obvious and why further investigations into the computational usefulness of orbital graphs should be pursued.

\begin{ex}\label{2K3}
We let $G:=\s_9$ and we look at the subgroup

$H:=\langle
(1\,2),(1\,3),(4\,5),(4\,6),(1\,4)(2\,5)(3\,6),(7\,8\,9)\rangle$. Let $\Gamma$ be the
orbital graph for $H$ with base-pair $(1,2)$. Then $\Gamma$ has the following shape:

On the vertices $1,2,3$ and $4,5,6$ we have a complete digraph, respectively,
there is no arc between the sets $\{1,2,3\}$ and $\{4,5,6\}$, and the points
$7,8$ and $9$ are isolated. This might look like a futile graph, but according
to the theorem it is not. Consider an ordered orbit partition
$P:=[1,2,3,4,5,6 \mid 7,8,9]$ of $H$.

The group $\Sym(P)$ contains the transposition $(2\,4) \in G$.
This element interchanges the vertices $2$ and $4$ of $\Gamma$ and fixes $1$, so
this element does not induce an automorphism on $\Gamma$. (Otherwise the arc
$(1,2)$ would be mapped to the arc $(1,4)$, which does not exist). This graph
can be used to deduce, for example, that any element which swaps $1$ and $4$ must also
swap $\{2,3\}$ with $\{5,6\}$.

Hence $\Sym(P)$ does not act as a group of automorphisms on $\Gamma$ and we see that
$\Gamma$ is not futile.
\end{ex}

It is important that we can detect futile graphs easily, without having to build them
explicitly. We will now give a collection of lemmas that allow futile orbital graphs to
be detected using only information about orbits and stabilizers of a group, without
explicit construction of entire orbital graphs.

\begin{lem}\label{allplnontrans}
Suppose that Hypothesis \ref{orbhyp} holds and that
\(\Omega=\alpha^H\dot{\cup}\beta^H\dot{\cup} I\), where $I \subseteq \Omega$ is the set of isolated vertices of $\Gamma$.
Then \(\Gamma\) is futile if and only if \(H_{\alpha}\) acts transitively on \(\beta^H\).

\end{lem}

\begin{proof}
Suppose that $\Gamma$ is futile. Then Theorem \ref{use} and Lemma \ref{basic}~(iii) and (iv) imply that
\(\Gamma\) is a complete bipartite digraph.
In particular, for all \(\delta \in \beta^H\) it follows that \((\alpha,
\delta)\in A\) and so there exists some \(h\in H\) such that
$(\alpha,\delta)=(\alpha^h,\beta^h)$. In particular $H_\alpha$ is transitive on
$\beta^H$.
Conversely we suppose that $H_\alpha$ is transitive on $\beta^H$. It follows that for all \(\beta' \in \beta^H\) there exists some \(h\in H_\alpha\) such that $\beta^h=\beta'$.

We prove that $H_\beta$ acts transitively on $\alpha^H$, so we let $\alpha' \in
\alpha^H$ and we choose $g \in H$ such that $\alpha'=\alpha^g$. Then, using the
transitivity argument above, we let $h \in H_\alpha$ be such that
$\beta^h=\beta^{g^{-1}}$, which implies $\beta^{hg}=\beta$ and $\alpha^{hg}=\alpha'$. Therefore $H_\beta$ acts transitively on $\alpha^H$.
Now the definition of an orbital graph implies that $\Gamma$ is a complete
bipartite digraph and hence futile, by Theorem \ref{use}.
\end{proof}

We finish this section by giving some concrete bounds on the number of edges in futile and non-futile orbital graphs.

\begin{lem}\label{altcounting}
  Suppose that Hypothesis \ref{orbhyp} holds. Let $n = |\alpha^H|$, $m =
  |\beta^H|$, and $I \subseteq \Omega$ be the set of isolated vertices of $\Gamma$.
  Then $\Gamma$ is futile if one of the following hold.
  \begin{enumerate}[(i)]
  \item $\beta \in \alpha^H$ and $\Gamma$ has strictly more than \(n(n-2)\)
    arcs.
  \item $\Omega = \alpha^H \dot\cup \beta^H \dot\cup I$ and $\Gamma$ has
    strictly more than \(n(m-1)\) or \(m(n-1)\) arcs.
  \end{enumerate}
\end{lem}

\begin{proof}
  To prove (i) suppose that \(\gamma,\delta \in \alpha^H\) are distinct and such
  that \((\gamma, \delta) \notin A\). Let \(r\) be the number of arcs starting in
  \(\gamma\). Now \((\gamma,\gamma)\) and \((\gamma,\delta)\) are not in $A$, so
  it follows that \(r\leq n-2\).
  We recall that $H$ is transitive on $\alpha^H$, and hence all connected
  components of $\Gamma$ have size at least $2$, by Remark \ref{rem}~(c).
In particular $\gamma$ is contained in a connected component of $\Gamma$ of
  size at least two, so we deduce from Lemma \ref{basic}~(iii) and (iv) and Lemma
  \ref{transimp} that for every vertex of \(\Gamma\), the number of arcs starting
  there is \(r\). Consequently $|A|= n\cdot r\leq n\cdot (n-2)$. This means,
  conversely, that $\Gamma$ is a complete digraph on $\alpha^H$ as soon as it has
  strictly more than \(n\cdot (n-2)\) arcs.

  To show (ii) suppose that there are \(\gamma \in \alpha^H\) and \(\delta \in
  \beta^H\) such that \((\gamma, \delta)\notin A\). Let \(r\) be the number of
  arcs starting in \(\gamma\). As all arcs starting in \(\gamma\) end in a vertex of
  \(\beta^G\backslash \{\delta\}\) it follows that \(r\leq m-1\).
  Let $\omega \in \alpha^H$. Then it follows from Lemma \ref{basic}~(iii) that the
  number of arcs starting in \(\omega\) is
  $|\{(\omega_1,\beta^{g})\mid g\in H, \alpha^g=\omega_1\}|$.
 Hence $|A|=n\cdot r\leq n\cdot (m-1)$.

  By counting the number of arcs ending in some vertex we obtain, in a similar way, that
  $|A| \le m\cdot (n-1)$ as well.
  Hence if \(\Gamma\) has strictly more than \(n\cdot (m-1)\) or \(m\cdot (n-1)\)
  arcs, then \(\Gamma\) is a complete bipartite digraph.
\end{proof}

In practice we use Corollary \ref{cor:futilecheck}, which combines Lemma \ref{altcounting} with Remark
\ref{rem} to efficiently identify futile
orbital graphs before they are constructed.

\begin{cor}\label{cor:futilecheck}
Suppose that Hypothesis \ref{orbhyp} holds. Then $\Gamma$ is
futile if and only if one of the following conditions is true:

\begin{enumerate}[(i)]
\item $\beta \in \alpha^H$ and $|\beta^{H_\alpha}| = |\alpha^H|$.
\item $\beta \not\in \alpha^H$ and $|\beta^{H_\alpha}| = |\beta^H|$.
\end{enumerate}
\end{cor}

\begin{proof}
\begin{enumerate}
\item[(i)] We are in Case (i) of Lemma \ref{altcounting}. By Remark \ref{rem} the orbital graph
has size  $|\alpha^H| \cdot |\beta^{H_\alpha}|$. The only way this can be larger than $|\alpha^H|(|\alpha^H|-2)$
is if $|\beta^{H_\alpha}| + 1 \geq |\alpha^H|$. As $\beta \in \alpha^H$, we see that $\beta^{H_\alpha}$ is a proper subset of $\alpha^H$ (the subset is proper because it does not contain $\alpha$). Therefore $|\beta^{H_\alpha}| + 1 = |\alpha^H|$.
\item[(ii)] We are in Case (ii) of Lemma \ref{altcounting}. Again by Remark \ref{rem} the orbital graph
has size $|\alpha^H| \cdot |\beta^{H_\alpha}|$. The only way this can be larger than $|\alpha^H|(|\beta^H|-1)$
is if $|\beta^H| \leq |\beta^{H_\alpha}|$. As $\beta^H$ contains $\beta^{H_\alpha}$, this implies that $|\beta^H| = |\beta^{H_\alpha}|$.
\end{enumerate}
\end{proof}

\begin{lem}\label{everypointless}
Suppose that Hypothesis \ref{orbhyp} holds and that $H$ acts transitively on $\Omega$.
\begin{enumerate}[(i)]
\item\label{ep1} If $H$ acts 2-transitively on $\Omega$, then $\Gamma$ is futile.
\item\label{ep2} If $\Gamma$ is futile, then $H$ acts 2-transitively on $\Omega$ (and hence all orbital graphs are futile).
\end{enumerate}
\end{lem}

\begin{proof}
For (\ref{ep1}) we suppose that \(H\) acts 2-transitively on \(\Omega\). Then
whenever \(\gamma, \delta \in \Omega\) are distinct, there exits some \(h \in H\) such that \((\alpha^h,\beta^h)=(\gamma,\delta)\) and hence
\(\Gamma\) is a complete digraph.
By Theorem \ref{use} it follows that $\Gamma$ is futile.

For (\ref{ep2}) we suppose that \(\Gamma\) is futile and we deduce, again by
Theorem \ref{use}, that \(\Gamma\) is a complete digraph or a complete bipartite
digraph. The second case is impossible because \(H\) is transitive on
\(\Omega\). So \(\Gamma\) is a complete digraph and for any two distinct
elements \(\gamma, \delta \in \Omega\), we deduce that \((\gamma, \delta) \in
A\). Then by definition of an orbital graph, there is \(h\in H\) such that
\((\alpha^h,\beta^h)=(\gamma,\delta)\). Hence \(H\) acts 2-transitively on
\(\Omega\) and the last statement follows from (\ref{ep1}).
\end{proof}

So we see that for transitive groups if one orbital graph is futile, then all of
them are. Lemma \ref{everypointless} lets us quickly detect this, as the level of transitivity of
a group can be efficiently calculated.

\subsection{Efficiently Creating Orbital Graphs}

While orbital graphs can be very useful in reducing search, they are expensive to create, so we only want to compute them when they provide extra refinements.
We use \cref{alg:orbitals} to compute orbital graphs, which assumes the use of a computational group theory system, such as GAP, that provides basic algorithms to compute point stabilizers, orbits of points, and 
orbits of pairs of points.

\begin{algorithm}
  \caption{Find Orbital Graphs}\label{alg:orbitals}
  \begin{algorithmic}[1]
    \Procedure{OrbitalBase}{$G,\Omega, SizeLimit$}\Comment{Orbital graphs of \(G\), a permutation group on \(\Omega\)}
    \State{\(\mathrm{Graphs} := []\)}
    \If{$G$ is \(k\)-transitive for \(k \geq 2\)}\label{line:2trans}
    \State \Return{Graphs}
    \EndIf
    \For{\(Orb \in \mathrm{Orbits}(G)\)}\label{line:outerorb}
    \If{\(|Orb| > 1\)}\label{line:if-basic}
    \State \(G' = \mathrm{Stabilizer}(G, \mathrm{Min}(Orb))\)\label{line:pickouterorb}
	  \For{\(InnerOrb \in\) Orbits\((G')\)}\label{line:innerorb}
    \If{\(|\mathrm{Orb}|\times |\mathrm{InnerOrb}| \leq SizeLimit\)}\label{line:if-sizelimit}
    \If{\(InnerOrb \not\in \mathrm{Orbits}(G)\)}\label{line:if-bipartite}
    \If{\(InnerOrb \in Orb\) and \(|InnerOrb| + 1 = |Orb|\)}\label{line:if-clique}
    \State \(\mathrm{Add}(Graphs, Orbit(G, (\mathrm{Min}(Orb), \mathrm{Min}(InnerOrb)) ) )\)
		\EndIf
    \EndIf\EndIf
	  \EndFor
    \EndIf
    \EndFor
    \State \Return{Graphs}
    \EndProcedure
  \end{algorithmic}
\end{algorithm}

The correctness of Algorithm \ref{alg:orbitals} is proven by applying results from the preceding sections, in particular \cref{lem:orbitalorbits}, \cref{everypointless}, and \cref{use}.

\begin{theorem}
Algorithm \ref{alg:orbitals} returns all orbital graphs, except those that are futile or that contain more than \textit{SizeLimit} edges.
\end{theorem}

\begin{proof}
First, \cref{line:2trans} performs an initial check whether the group is
\(k\)-transitive for \(k \geq 2\). If it is, then we stop because, by Lemma \ref{everypointless}(i), all orbital graphs of \(G\) are futile in this case.

After this, \cref{line:outerorb} picks one member \(a\) from each orbit of \(G\)
to be the first member of a base-pair. Here we apply Lemma \ref{lem:orbitalorbits}: In order to generate all orbital graphs, 
it is enough to pick one element from each orbit as first member, by Part (i) of the lemma. Then
Part (ii) shows that no orbital graph will be generated twice.
\cref{line:pickouterorb,line:innerorb} pick one member from each of the orbits of the stabilizer of each of our first base-pair points. Part (iii) of the Lemma \ref{lem:orbitalorbits} shows that this will give us a set of base-pairs from which every orbital graph arises exactly once.
We now move through the other lines, which skip orbital graphs we do not want to
consider. \Cref{line:if-bipartite} and \Cref{line:if-clique} check the
conditions of Theorem \ref{use}, rejecting all futile orbital graphs.
\Cref{line:if-sizelimit}  allows us to reject any orbital graph that exceeds a user-defined
limit on the number of edges.

Finally, \cref{line:if-basic} provides a fast early check. If some \(\alpha \in
\Omega\) is already fixed by \(G\), then \(G=G_\alpha\), which means that
\Cref{line:if-bipartite} will always fail. Therefore we may as well reject such
points immediately.
\end{proof}

The runtime of \cref{alg:orbitals} is, for most problems, dominated by the
calculation of the point stabilizers on \Cref{line:pickouterorb}. In GAP these
are calculated using a randomized implementation of the Schreier-Sims algorithm.
As our experiments will show, \cref{alg:orbitals} performs well in practice.

\section{Graph refiners}

We will employ existing refiners for arbitrary graphs, as
discussed in \cite{McKay80} and \cite{McKay201494},
to create refiners for groups given as a set of generators.

Before we go into the details, we need some more definitions:

\begin{definition}[Equalizer]
Let $\Gamma = (\Omega,A)$ be a digraph.

For any ordered partition $P \in \OPart(\Omega)$, we say that a cell $\Delta$ of $P$ is
\textbf{$\Gamma$-equitable} if, for all cells $\Delta'$ of
$P$ there is some $k \in \N$ such that for all elements $i$ in the cell $\Delta$ it
holds that $|\{j \in \Delta' \mid (i,j) \in A$ or $(j,i) \in A \}| = k$.

We note that in this definition the number of arcs $k$ depends on $\Delta$ and
$\Delta'$, but not on the individual vertices in $\Delta$.

An ordered orbit partition $P$ is called \textbf{$\Gamma$-equitable} if all its cells are
$\Gamma$-equitable.\\

Next we suppose that $P$ and $P'$ are ordered partitions of $\Omega$.
We say that $P'$ is a \textbf{$\Gamma$-equalizer for $P$} if and only if $P'$ is
$\Gamma$-equitable, moreover $P'\preccurlyeq P$ and $P'$ is as coarse as
possible with this property, meaning that whenever $Q \in \OPart(\Omega)$ is also
$\Gamma$-equitable and $Q \preccurlyeq P$, then $Q \preccurlyeq P'$.
\end{definition}

\begin{rem}
Some remarks on the previous definition:

If $\Gamma = (\Omega,A)$ is a digraph and $P\in\OPart(\Omega)$, then two
distinct $\Gamma$-equalizers for $P$ can only differ by the ordering of their cells.
We also note that, if $P$ is $\Gamma$-equitable and $g$ is an automorphism of
$\Gamma$, then $P^g$ is also $\Gamma$-equitable.
\end{rem}

\cite{McKay201494} present algorithms for calculating $\Gamma$-equitable
ordered partitions. We discuss a simple algorithm that computes a
$\Gamma$-equalizer for $P$, given a digraph $\Gamma$ and an ordered orbit partition $P$ as input. The main difference between the algorithm in \cite{McKay201494}, and our implementation,
is that McKay and Piperno optimize this algorithm by showing several cases where
they can skip some attempts to split because 
they know that no splitting will occur. We omit these improvements because the total
time taken by the refinement algorithm in our partition backtracker is usually
negligible.

\begin{algorithm}
  \caption{Equitable Partitions}\label{alg:equitable}
  \begin{algorithmic}[1]
    \Procedure{Equitable}{$\Gamma, P$}
    \State{$\overline{P} := P$}
    \State{$T := P$}
    \While{($T$ not empty) \textbf{and} ($\overline{P}$ is not discrete)}
    \State{Pick and remove some cell $\Delta \in T$}
    \For{$\Delta' \in \overline{P}$}
    \State{Split $\Delta'$ into $\Delta'_1 \dots \Delta'_k$ equitably, according to
      edges starting at vertices in $\Delta$}
    \If{$k > 1$}
    \State{Replace the cell $\Delta'$ in $\overline{P}$ with $\Delta'_1 \dots \Delta'_k$}
    \State{Add $\Delta'_1,\ldots,\Delta'_k$ to $T$}
    \EndIf
    \EndFor
    \EndWhile
    \State\Return{$\overline{P}$}
    \EndProcedure
  \end{algorithmic}
\end{algorithm}

Using orbital graphs and ordered equitable partitions, we define our new refiner.

\begin{definition}[$\Orb_M$]\label{def:orbrefine}
   Given $M$ a subgroup of $\Sym({\Omega})$, the map $\Orb_M: \OPart(\Omega) \rightarrow \OPart(\Omega)$ is defined
   as follows:

  \begin{itemize}
  \item Construct all orbital graphs of $M$.
  \item Given an ordered orbit partition $P$ of $\Omega$, compute a $\Gamma$-equalizer for $P$
    for every orbital graph $\Gamma$ from the previous step, using
    Algorithm~\ref{alg:equitable}.
  \item Return the meet of all refined ordered partitions from the previous step.
  \end{itemize}

  We use the notation $\Orb$, without a subscript, for refiners for subgroup search.
\end{definition}

We argue now that $\Orb_M$ is in fact a refiner. A detailed example later in this section will illustrate how refiner works.

\begin{lem}\label{Orbref}
  $\Orb_M$ is a refiner.
\end{lem}

\begin{proof}
  For all $P \in \OPart(\Omega)$, it holds that $\Orb_M(P)\preccurlyeq P$,
  because Algorithm~\ref{alg:equitable} splits cells of $P$.
  Next, we need to show that for all $g \in M$ it holds that $\Orb_M(P^g) =
  {Orb_M(P)}^g$, but this follows directly from the fact that $g$ is an
  automorphism of any orbital graph, and that $g$ commutes with taking meets
  of ordered partitions by Lemma \ref{lem:meetandact}.
\end{proof}

One obvious limitation of $\Orb$ is that, if the group is 2-transitive, then it
does not perform any refinement, as the only orbital graph is the complete graph
on \(\Omega\) (see Lemma \ref{everypointless}). We therefore introduce another orbital graph based refiner that
makes use of the fact that we can, like in $\Fixed$, easily stabilize points in
the group.

\begin{definition}[$\DeepOrb_M$]\label{def:deeporbrefine}
   Given $M$ a subgroup of $\Sym({\Omega})$, the map $\DeepOrb_M: \OPart(\Omega) \rightarrow \OPart(\Omega)$ is
   defined as follows:

  \begin{itemize}
  \item Given $P \in \OPart(\Omega)$, let $k \in \N$ and
    $\alpha_1,\ldots,\alpha_k \in \Omega$ be the elements in singleton cells of $P$.
    Then let \(M_0\) denote the point-wise stabilizer of $\alpha_1,\ldots,\alpha_k$ in $M$.
  \item Construct all orbital graphs of $M_0$.
  \item Given an ordered orbit partition $P$ of $\Omega$, compute a $\Gamma$-equalizer for $P$
    for every orbital graph $\Gamma$ from the previous step, using
    Algorithm~\ref{alg:equitable}.
  \item Return the meet of all refined ordered partitions from the previous step.
  \end{itemize}

  We use the notation $\DeepOrb$, without a subscript, for this refiner.
\end{definition}

$\DeepOrb$ can be seen as combining $\Orb$ and $\Fixed$. The proof of
correctness of $\DeepOrb$ is analogous to the proof for $\Orb$. The major
disadvantage of $\DeepOrb$ is that it requires calculating the orbital graphs at
every level in the search, rather than just once at the beginning. Our
experiments will investigate the practical trade-offs between $\Orb$ and $\DeepOrb$.

\subsection{Examples of refinements via graphs}

We will now present two examples of refinement in which the \(\Fixed\)
refiner will perform no refinement at all, but orbital graphs provide useful refinement.

We follow \cite{McKay201494}, referring to partitions as colourings. We will refer interchangeably to colouring vertex $j$ of the graph with colour $\delta$ and placing value $j$ into cell $\delta$ of the ordered partition.

In the following two examples we let $G:=\s_{10}$.
\begin{ex}
Let
$H_1:=\langle (1,2,...,10),(2,10)(3,9)(4,8)(5,7)\rangle$ and let $H_2$
denote the stabilizer of the set \(\{1,5\}\) in $G$. 
We are interested in calculating $D:=H_1 \cap H_2$, which is equivalent to 
calculating the stabilizer of \(\{1,5\}\) in $H_1$. 
While this problem is very simple, it allows us to show in detail how the
algorithm works.

Purely group-theoretically, if we take $x \in D$, then all we know without further calculation is that $x$
stabilizes the set $\{1,5\}$ and so we obtain the ordered partition
$P_1:=[1,5 \mid 2, 3, 4, 6, 7, 8, 9, 10]$.

Looking at the orbits of \(H_1\) produces no useful information, as \(H_1\) is
transitive and therefore \(\Fixed_{H_1}(P_1)=P_1\). Therefore the only
information that we extract from reasoning about orbits alone is that $D$ is
contained in a subgroup of $G$ that is isomorphic to $\Sym(P_1)$.

Now we use the orbital graph for $H_1$ with base-pair $(1,2)$. For simplicity we work with an undirected graph here because all
arcs exist in both directions. When calculating equitable partitions, we will only
show steps where the algorithm causes a cell of an ordered orbit partition to split, skipping
steps where no split occurs.

\textbf{Step 1:}

Using \(P_1\), we attach the colour $1$ to the vertices $1$ and $5$ (first cell of the partition) and colour $2$ to all other vertices (second cell of the partition). 

\begin{center}
\begin{tikzpicture}[shorten >=1pt,-]
  \tikzstyle{vertex}=[circle,fill=black!25,minimum size=14pt,inner sep=0pt]

  \foreach \name/\angle/\text/\col in {
    P-1/108/1/black!50,
    P-2/144/2/black!5,
    P-3/180/3/black!5,
    P-4/216/4/black!5,
    P-5/252/5/black!50,
    P-6/288/6/black!5,
    P-7/324/7/black!5,
    P-8/0/8/black!5,
    P-9/36/9/black!5,
    P-10/72/10/black!5 }
  \node[vertex] (\name) at (\angle:1.75cm) [fill=\col] {$\text$};

  \foreach \from/\to in {1/2,2/3,3/4,4/5,5/6,6/7,7/8,8/9,9/10,10/1}
    { \draw (P-\from) -- (P-\to); }
\end{tikzpicture}
\end{center}

\textbf{Step 2:}

Using \cref{alg:equitable}, we look at the vertices of colour $2$. They fall into two classes -- those who have two neighbours of colour $2$, and those who have a neighbour of colour $1$ and a neighbour of colour $2$. This splits the second cell of $P_1$ into two cells, with $2,4,6$ and $10$ in a new cell, with colour $3$.
Our second ordered partition is therefore
$P_2:=[1,5 \mid 3, 7, 8, 9 \mid 2, 4, 6, 10]$ and $D$ is isomorphic to a subgroup of $\Sym(P_2)$.\\

\textbf{Step 3:}

Continuing the algorithm, we see that the vertices of colour $2$ can be further divided into a single vertex that has two neighbours of colour $2$ (vertex 3), those that have two neighbours of colour $3$ (vertex 8) and those that have one neighbour of colour $2$ and one of colour $3$ (vertices 7 and 9). With this information we obtain the new ordered partition
$P_3:=[1,5 \mid 7,9 \mid 2, 4, 6, 10 \mid 3 \mid 8]$,  and $D$ is isomorphic to a subgroup of $\Sym(P_3)$.\\

\textbf{Step 4:}

Our algorithm continues, looking at cell $3$. Here the vertices can be divided into two categories, namely those with a neighbour of colour $1$ and a neighbour of colour $4$ (vertices 2 and 4), and those with a neighbour of colour $1$ and a neighbour of colour $2$ (vertices 6 and 10).
Our final ordered partition is therefore
$P_4:=[1,5 \mid 8 \mid 6,10 \mid 3\mid 7, 9 \mid 2, 4]$.\\

There is no further information in the graph at the moment, so we conclude by stating that $D$ is isomorphic to a subgroup of
$\Sym(P_4)$, which has order $16$. 

Given that $P_4$ is an ordered orbit partition for \(H_1 \cap H_2\), it is no surprise that no further refinement is possible. 
We also know that $D \le H_1$ and $|H_1|=20$, so
$|D| \le 4$.
\end{ex}

In this case, we were able to deduce the exact orbits of \(H_1 \cap H_2\). This is not true in general, as our next example will show. But we will still perform useful deductions using the orbital graph that cannot be performed by  \(\Fixed\).

\begin{ex}
Again $H_1:=\langle (1,2,...,10),(2,10)(3,9)(4,8)(5,7)\rangle$, but this time $H_2$ is the stabilizer of the set $\{1,6\}$ in $G$. Again we wish to calculate $D:=H_1 \cap H_2$.
Using reasoning from the orbits of \(H_2\) we obtain the ordered partition
$Q_1=[1,6
\mid 2, 3, 4, 5, 7, 8, 9, 10]$. We note that
\(\Fixed_{H_2}(Q_1)=Q_1\) because \(H_2\) is transitive
and \(Q_1\) has no singleton cells.
We consider the same orbital graph as in the previous example.\\

\textbf{Step 1:}

We create a graph where we attach the colour $1$ to the vertices $1$ and $6$,
and colour $2$ to all other vertices. This gives the ordered partition $[1,6 \mid 2, 3,
4, 5, 7, 8, 9, 10]$.\\

\textbf{Step 2:}

We split the vertices in the second cell into vertices that have different
coloured neighbours (vertices 2,7,5 and 10) and those with two neighbours of
colour $2$ (vertices 3,4,8 and 9).
Our second ordered partition is therefore
$Q_2 = [1,6 \mid 3, 4, 8, 9 \mid 2, 7, 5, 10]$.\\

\textbf{Step 3:}

The algorithm will now run through the remaining reasoning, not producing any
further refinements, as the ordered partition is already equitable.
We can say that $D$ is isomorphic to a subgroup of $\Sym(Q_2)$ of order $2^7
\cdot 3^2$. Since $D$ also is a subgroup of $H_1$ which has order $20$, we
deduce that $|D| \le 4$.

Putting together both of our examples, using orbital graphs we know that any
elements of \(H_1\) in \(\Co(P_1,Q_1)\) are also in \(\Co(P_2,Q_2)\), by
Lemma~\ref{lem:backtrack}. However, \(Co(P_2,Q_2) = \varnothing\), because \(P_2\)
and \(Q_2\) have different numbers of cells, and therefore we have deduced there
are no members of \(H_1\) in \(\Co(P_1,Q_1)\).
\end{ex}

\section{Experiments}

We will now demonstrate how our algorithm performs on a variety of
problems. There are two main questions we want to answer:

\begin{enumerate}
\item How much can we speed up the search?
\item What are the worst-case slowdowns when using orbital graphs?
\end{enumerate}

Generating a set of ``random'' problems -- taken from all possible group
intersection and stabilizer problems -- is not possible because there is no method of
producing random permutation groups of large size. Also, we do not want to
consider problems that are very simple -- partition backtrack solves many
problems almost instantly, with or without orbital graphs.
We consider two main categories of problems: The first one is calculating set stabilizers in ``grid groups'' that arise in A.I. in SAT solvers (the Boolean satisfiability problem) and Constraint Programming and are used for symmetry-breaking, and the second one is
taking intersections of primitive groups with wreath products.
We discuss the following four algorithms:

\begin{enumerate}
\item \textbf{Fixed}: $\Fixed$, the traditional algorithm of Leon.
\item \textbf{PreOrbital}: Use $\Fixed$ with $\Orb$ as described in Section 4.
\item \textbf{DeepOrbital}: Use $\Fixed$ with $\DeepOrb$ as described in Section 4.
\item \textbf{FirstOrbital}: Use $\Fixed$ and a variant of $\DeepOrb$ that
keeps building orbital graphs until the first refinement phase at which at least one non-futile orbital graph
is found. Then these orbital graphs are used in all later refinement phases.
\end{enumerate}

\textbf{FirstOrbital} is implemented as a variant of  \textbf{DeepOrbital}.
\textbf{FirstOrbital} is difficult to cleanly specify theoretically, as it is tied to
behaviour of the search.
\textbf{FirstOrbital} is a valid refiner in practice, because down
each branch of search we will find the first node with at least one non-futile orbital
graph at the same height. 
This is a consequence of the definition of the $\Fixed$ refiner -- it ensures that there is an element of the group
that maps the fixed points down the first branch to the fixed points down every other branch at every
level.

We expect that \textbf{Fixed} will always produce larger search trees than
\textbf{PreOrbital}, which in turn produces larger search trees than
\textbf{FirstOrbital}, which will produces larger search trees than \textbf{DeepOrbital}.
In some rare cases the searches can be larger with a better refiner, as the partition backtrack
algorithm may choose to branch on a different cell, or on the elements of a cell in a different order.
In practice this occurred in less than 2\% of our experiments.

The purpose of these experiments is to show when the decreased search size
provided by \textbf{PreOrbital}, \textbf{DeepOrbital} and \textbf{FirstOrbital}
outweighs the increased cost of calculating and filtering orbital graphs.

All of our experiments are performed in GAP (see ~\cite{GAP4}). We use the
implementation of partition backtrack provided in the Ferret package (see
\cite{Fer}), which includes both an implementation of Leon's original partition backtrack
algorithms, and our new algorithms. In the experiments in this paper,
Ferret's implementations of Leon's algorithms are always faster than the
implementations in GAP, due to improved quality of implementation and the data structures used.

All of our experiments were performed on a machine with eight Intel Xeon E5520
cpus, running at 2.27GHz and 20GB RAM.
Each experiment was given a five minute timeout and a limit of 1GB of RAM.

\subsection{Set stabilizers in Grid Groups}

A typical example for a problem that involves grid groups is shift scheduling. 
If we have $n$ workers and $m$ time-slots, where the workers are interchangeable (in terms of their qualification) and the time-slots are equally important, then this symmetry in the system can be expressed via a grid group.

\begin{definition}[Grid Group]\label{def:gridgroup}
  Let $n \in \N$. The direct product $\s_m \times \s_m$ acts on the set
  $\{1,\ldots, m\} \times \{1,\ldots, m\}$ of pairs in the following way:

  For all $(i,j) \in  \{1,\ldots, m\} \times \{1,\ldots, m\}$
  and all $(\sigma,\tau) \in \s_m \times \s_m$ we define

  \[
    {(i,j)}^{(\sigma,\tau)} := (i^{\sigma}, j^{\tau}).
  \]

  The subgroup $G \le \Sym(\{1,\ldots, m\} \times \{1,\ldots, m\})$ defined by this
  action is called the \textbf{$m \times m$ grid group}.
\end{definition}

While the construction of the grid group is done by starting with an $m$
by $m$ grid of points and permuting rows or columns independently of each other,
we represent this group as a subgroup of $\s_{m^2}$, and we do not
assume prior knowledge of the grid structure of the action.
We considered two different variants of set stabilizers: stabilizing a random
subset of $\{1, \ldots, m^2\}$ of size $\lfloor \frac{m^2}{2} \rfloor$, and stabilizing
random subsets of $\{1, \ldots, m^2\}$ containing exactly $\lfloor \frac{m}{2}
\rfloor$ points in each row.
The specific choice of sets in the second variant leads to more difficult problem, intuitively, because it is harder to prove that the stabilizer does not permute the rows of the grid.

\begin{table}
\begin{center}
\begin{tabular}{|r|r|r|r|r|r|r|r|r|r|}
\hline
 & \multicolumn{2}{|c|}{\textbf{Fixed}} & \multicolumn{2}{|c}{\textbf{PreOrbital}} &  \multicolumn{2}{|c|}{\textbf{DeepOrbital}} \\
$m$ & \# &  Time & \# &  Time  & \# &  Time  \\
\hline
10 & 10  & 0.1 & 10  & 0.1 & 10  & 0.5 \\
20 & 10  & 2.0 & 10  & 0.5 & 10  & 10.8 \\
30 & 10  & 15.6 & 10  & 2.9 & 10  & 79.4 \\
40 & 10  & 108.6 & 10  & 10.5 & 10  & 262.9 \\
50 & 7  & 1,506.5 & 10  & 40.2 & 10  & 643.1 \\
60 & 3  & 2,554.5 & 10  & 101.9 & 6  & 2,253.7 \\
70 & 0 & 3,000 & 9  & 490.6 & 2  & 2,862.1 \\
80 & 0 & 3,000 & 10  & 497.8 & 4  & 1,951.4 \\
90 & 1  & 2,772.3 & 10  & 1,052.7 & 0 & 3,000 \\
100 & 0 & 3,000 & 9  & 1,195.2 & 4  & 2,009.1 \\
110 & 0 & 3,000 & 9  & 1,554.3 & 5  & 1,869.7 \\
120 & 0 & 3,000 & 0 & 3,000 & 0 & 3,000 \\
\hline
\end{tabular}
\end{center}
\caption{Time taken to find set stabilizers in grid groups (average of 10
runs). ``\#'' is the number of problems solved. ``Time'' is the total time in seconds to complete (or timeout) the 10 instances.\label{fig:type1grid}}
\end{table}

\begin{table}
\begin{center}
\begin{tabular}{|r|r|r|r|r|r|r|r|r|r|}
\hline
 & \multicolumn{2}{|c|}{\textbf{Fixed}} & \multicolumn{2}{|c}{\textbf{PreOrbital}} &  \multicolumn{2}{|c|}{\textbf{DeepOrbital}} \\
$m$ & \# & Time & \# & Time  & \#  & Time  \\
\hline
10 & 10  & 0.1 & 10  & 0.1 & 10  & 0.3 \\
20 & 4  & 1,892.8 & 10  & 0.4 & 10  & 0.9 \\
30 & 3  & 2,107.3 & 10  & 2.1 & 10  & 5.1 \\
40 & 0 & 3,000 & 10  & 8.1 & 10  & 8.2 \\
50 & 0 & 3,000 & 10  & 30.3 & 10  & 30.3 \\
60 & 0 & 3,000 & 10  & 87.3 & 10  & 84.5 \\
70 & 0 & 3,000 & 10  & 235.0 & 10  & 236.2 \\
80 & 0 & 3,000 & 10  & 542.0 & 9  & 802.2 \\
90 & 0 & 3,000 & 10  & 1,327.4 & 9  & 1,393.1 \\
100 & 0 & 3,000 & 3  & 2,589.7 & 4  & 2,553.4 \\
110 & 0 & 3,000 & 2  & 2,919.0 & 1  & 2,941.8 \\
120 & 0 & 3,000 & 1  & 2,880.9 & 1  & 2,890.2 \\
\hline
\end{tabular}
\end{center}
\caption{Time taken to find row-balanced set stabilizers in grid groups
(average of 10 runs). ``\#'' is the number of problems solved.
``Time'' is the total time in seconds to complete (or timeout) the 10 instances.\label{fig:type2grid}}
\end{table}

First, we generate a random set of size \( \lfloor \frac{m^2}{2} \rfloor \) of
\(\{1 \ldots m^2\}\). The results are shown in Table~\ref{fig:type1grid}. For times, we use the timeout time (5
minutes) for instances that ran out of either time or memory. We do not display
results for \textbf{FirstOrbital}, because they are identical to
\textbf{PreOrbital}.
Our problems are randomly generated and therefore some are simpler than others, which is
why we see that \textbf{Fixed} is able to solve one problem on a grid of size
90. However, for all sizes of grids we see a substantial improvement from building orbital
graphs. As
\textbf{DeepOrbital} keeps rebuilding the graphs whenever a new point is fixed
in the ordered partition, the time taken is generally longer, as the extra graphs do not
pay back their cost for most of these problems.
In the second experiment, we generate row-balanced sets. These include exactly
$\lfloor \frac{m}{2} \rfloor$ points in each for row of the grid. 
As mentioned earlier, we expect these instances to be more difficult.
The results of this experiment are shown in Table~\ref{fig:type2grid}. The
major difference here is that no instances of size bigger than 30 were solved by
\textbf{Fixed}, while \textbf{PreOrbital} and \textbf{DeepOrbital} solved almost
all instances. \textbf{FirstOrbital} once again had identical results to
\textbf{PreOrbital}, so we skip this step.

\begin{table}
\begin{center}
\begin{tabular}{|l|l|r|r|r|r|r}
\hline
Problem&Refiner & Building & Building & Refining \\
& & Stab Chains & Orbital Graphs & Orbital Graphs\\
\hline
& \textbf{Fixed} & 39.4 & 0 & 0\\
Set Stab& \textbf{PreOrbital} & 95.2 & 1.1 & 0.1\\
& \textbf{DeepOrbital} & 25.5 & 58.5 & 0.4\\
\hline
Row Balanced& \textbf{Fixed} & 1.3 & 0 & 0\\
Set Stab& \textbf{PreOrbital} & 97.8 & 0.9 & 0.1\\
& \textbf{DeepOrbital} & 97.5 & 1.0 & 0.1\\
\hline
\end{tabular}
\end{center}
\caption{Percentage of runtime spent on stabilizer chains and orbital graphs in Grid set stabilizer experiments.\label{fig:breakdowngrid}}
\end{table}

In both experiments, we also recorded the amount of time spent performing
different parts of the search. These results are shown in
Table~\ref{fig:breakdowngrid}. We see that \textbf{PreOrbital} spends around 1\%
of the total time both building and refining orbital graphs. While a much higher
proportion of time is spent building stabilizer chains, these are chains are a
subset of the chains which had the be built by \textbf{Fixed}. Further, we investigated the
size of the searches produced, and found that in almost every problem \textbf{DeepOrbital}
and \textbf{PreOrbital} were able to find the stabilizer (which was the trivial group) without
performing any branching. Therefore, the building of
these stabilizer chains is now the limiting factor to any further improvement.
We see that \textbf{DeepOrbital} behaves very strangely, spending much more time
building graphs with non row-balanced sets than with row-balanced sets. These
results occur because during refinement the non row-balanced problems take more
refinement steps to reach a fixed ordered partition. This results in more graphs being
created. We conclude that the performance of \textbf{DeepOrbital} can be very
unpredictable.

Of the 240 experiments, for
\textbf{Fixed}, 68 instances finish, 151 timeout and 21 run out of memory. For
\textbf{PreOrbital}, 203 instances finish, 37 run out of time and no instance
runs out of memory. For \textbf{DeepOrbital}, 165 instances finish, 51 run out of
time and 24 run out of memory.

The main observation of this experiment is that the overhead of
\textbf{PreOrbital} is very small while there is an exponential decrease in
search size. While \textbf{DeepOrbital} did not take much more time, it did take
much more memory. Further, its behaviour is unpredictable. The behaviour of
\textbf{FirstOrbital} is identical to \textbf{PreOrbital} on this problem. The
limiting factor is now the time taken finding stabilizer chains, which is
outside the scope of this article.

% !TEX root = GraphsReallyHelpRevised.tex

\subsection{Intersection of Primitive Groups}

In our second experiment, we consider intersection of groups. We concentrate on primitive groups because the primitive group library in GAP
provides easy access to a large set of groups.
First we intersected pairs of primitive groups, but then it turned out that this problem is
usually extremely fast to solve, with or without orbital graphs. Therefore we instead intersect a
primitive group with a wreath product of symmetric groups, and we found that this
produced challenging problems.

For each \(n \in \N\) we take all primitive groups except the symmetric groups
\(S_n\), the alternating groups \(A_n\), the cyclic groups \(C_n\) and the
dihedral groups \(D_n\) in their natural action on \(n\) points. We remove these
groups, because GAP handles them as a special case, and intersections
involving these groups are  simple to construct.
All 2-transitive groups are primitive and therefore we will encounter many 2-transitive groups in this experiment. However, we already know from Lemma \ref{everypointless} that they have futile orbital graphs. Therefore we consider 2-transitive groups separately.

Given a primitive group \(G\) acting on \(n\) points, we create the wreath
products \(S_{n/x} \wr S_x\), where \(x \in \{2,\ldots,7\}\) and \(n/x\) is an
integer. We conjugate each of these wreath products by a random element of \(S_n\)
and intersect the result with \(G\). We experimented with other wreath products
and found similar results.
With primitive groups on up to 600 points, we produce a total of
2,752 experiments on primitive groups that are not 2-transitive and 1,140
experiments for 2-transitive groups. We ran each experiment for a maximum of five minutes.
Note that we build the orbital graphs for both the primitive groups and the
wreath products. The orbital graphs on wreath products are often very large.
There are two orbital graphs of the wreath product of \(S_a \wr S_b\), one
consists of \(a\) cliques of size \(b\), and the second all the other edges
of the graph.

Looking firstly at groups which are primitive but not 2-transitive,  \textbf{Fixed} solved 1,794 problems within the timeout,
\textbf{PreOrbital} solved 2,410, \textbf{DeepOrbital} solved 2,377 and \textbf{FirstOrbital} solved 2,411.
We show the cumulative time taken to solve (or timeout) all problems in Table~\ref{fig:graph2prim}. Here we can see that all the orbital techniques solve problems on average much faster. In particular, the many primitive groups on 256 points cause severe difficulties for the \textbf{Fixed} refiner. There is not a similar spike for 512 points, because we only include problems at least one algorithm solved within the timeout.
Our results show that all techniques involving orbital graphs solve substantially more problems within the timeout,
and the problems solved were
solved much faster with orbital graphs. \textbf{DeepOrbital} is slightly slower but not
significantly so, and \textbf{PreOrbital} and \textbf{FirstOrbital} are almost
identical, which we would expect.
For 2-transitive groups
\textbf{Fixed} solved 992 problems, \textbf{PreOrbital} solved 982,
\textbf{DeepOrbital} solved 1,083 and \textbf{FirstOrbital} solved 1,071.
We show the cumulative time taken to solve (or timeout) all problems in Figure~\ref{fig:graph2prim2}.
Here we find orbital graphs are less useful, but still allow us to solve more problems within the timeout,
and to solve problems within the timeout faster. \textbf{PreOrbital} does not
pay off the cost of building orbital graphs, taking significantly longer.
\textbf{DeepOrbital} solves more problems, and solves them faster. \textbf{FirstOrbital} solves the
most problems, taking only slightly longer than \textbf{DeepOrbital}.

\begin{table}
\begin{center}
\begin{tabular}{|l|l|r|r|r|r|r}
\hline
&&Orbital    & Orbital & Orbital & Orbital \\
&& much slower& slower  & faster  & much faster\\
\hline
Not& \textbf{PreOrbital}      & 13  & 157 & 149 & 1519 \\
2-trans& \textbf{DeepOrbital} & 297 & 191 &  77 & 1358 \\
& \textbf{FirstOrbital}       &  12 & 146 & 162 & 1519 \\
\hline
& \textbf{PreOrbital}         & 114 & 526 & 63 & 0\\
2-trans& \textbf{DeepOrbital} & 38 & 197 & 206 & 351\\
& \textbf{FirstOrbital}       & 33 & 176 & 219 & 346\\
\hline
\end{tabular}
\end{center}
\caption{The relative performance, compared to \textbf{Fixed}, on Primitive Intersection experiments.\label{fig:RelativeSpeed}}
\end{table}

\begin{table}
\begin{center}
\begin{tabular}{|l|l|r|r|r|r|r}
\hline
Problem& Refiner& Build & Build & Refine \\
& & Stab Chains & Orbital Graphs & Orbital Graphs\\
\hline
& \textbf{Fixed} & 99.7 & 0 & 0\\
Not& \textbf{PreOrbital} & 97.5 & 0.5 & 0.1\\
2-trans& \textbf{DeepOrbital} & 77.5 & 7.2 & 10.2\\
& \textbf{FirstOrbital} & 97.5 & 0.5 & 0.1\\
\hline
& \textbf{Fixed} & 80.8 & 0 & 0\\
& \textbf{PreOrbital} & 78.3 & 0.1 & 4.7\\
2-trans& \textbf{DeepOrbital} & 86.7 & 3.6 & 6.1\\
& \textbf{FirstOrbital} & 88.9 & 0.6 & 6.6\\
\hline
\end{tabular}
\end{center}
\caption{Percentage of runtime spent building stabilizer chains and orbital graphs in Primitive Intersection experiments.\label{fig:breakdownprim}}
\end{table}

In Table~\ref{fig:RelativeSpeed} we compare the algorithms $\textbf{PreOrbital}, \textbf{DeepOrbital}$ and $\textbf{FirstOrbital}$ to $\Fixed$.
The column ``much faster'' captures instances where the algorithms involving orbital graphs are at least two times faster than $\Fixed$,
and similarly the column ``much slower'' contains cases where orbital graphs are at least two times slower.
Table~\ref{fig:breakdownprim} shows how much time was spent in different parts of
the algorithm during the search. As with the grid experiments, the majority of
time was spent finding stabilizer chains. The \textbf{PreOrbital} and
\textbf{FirstOrbital} algorithms both spend very little time building and
refining with orbital graphs. The search sizes are significantly
reduced, for groups acting on over 500 points \textbf{Fixed} averages around 3.7
million nodes on problems it can solve within the timeout, while
\textbf{FirstOrbital} averages 21,000 nodes on the same set of problems.
Similar to the earlier grid experiments, the main limiting factor in further performance
improvements is the building of stabilizer chains.

\subsection{Conclusions}

Our experiments show that using orbital graphs in partition backtrack can lead
to substantial performance improvements -- in the case of grid groups we can
easily solve much larger problems than before. Here the only limit is how
quickly we can
calculate stabilizer chains. While \textbf{DeepOrbital} is still a huge improvement
over \textbf{Fixed} alone, the overhead of calculating orbital graphs for many groups
is not recovered. Further improvements to the performance of set stabilizers in larger grids
would require either new methods of finding stabilizer chains, or fundamental changes to
partition backtrack to remove the need to calculate so many stabilizer chains.
For primitive groups, the results are more mixed. Here we can see that the cost
of building an orbital graph is larger -- wreath products of symmetric groups in particular
produce large orbital graphs. For groups that are not 2-transitive, all our orbital graph methods
perform similarly. For 2-transitive groups however, we see that \textbf{PreOrbital}, as expected, is
slightly slower than \textbf{Fixed}, but \textbf{FirstOrbital} and \textbf{DeepOrbital} perform well. There
is a small set of problems that only \textbf{DeepOrbital} is able to solve within the time limit.

The most important result to take away from our experiments is that \textbf{FirstOrbital} is
a balanced algorithm with good practical behaviour. In all our experiments it has a low overhead.
We suggest always using \textbf{FirstOrbital}, because often it is much better than \textbf{Fixed}, and it is close to the best of the algorithms on all our problem classes.

\begin{figure}
\begin{center}
\includegraphics[width=0.8\textwidth]{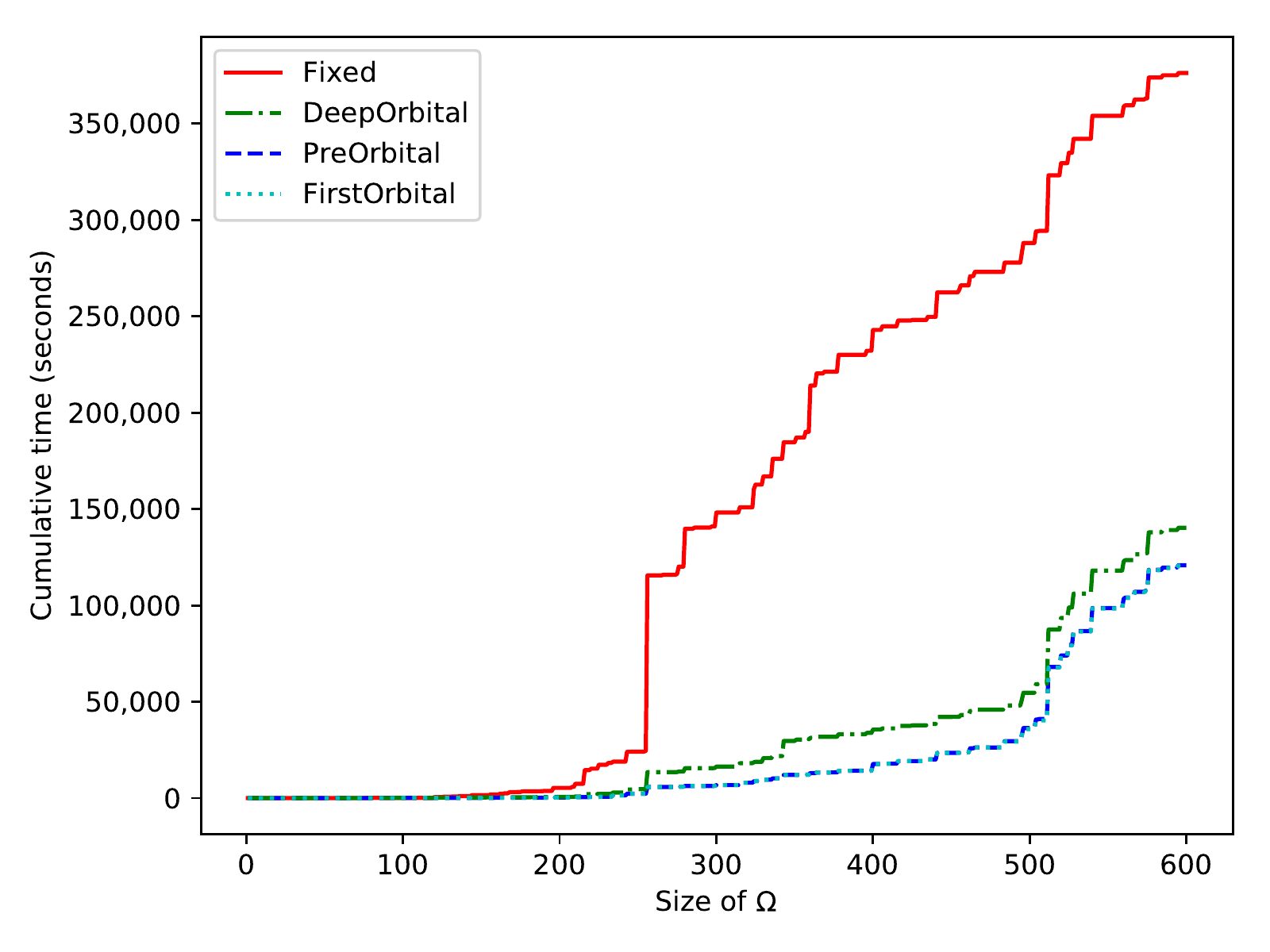}

\caption{Cumulative time taken to intersect primitive and not 2-transitive primitive groups with a wreath product\label{fig:graph2prim}}

\end{center}
\end{figure}

\begin{figure}
\begin{center}
\includegraphics[width=0.8\textwidth]{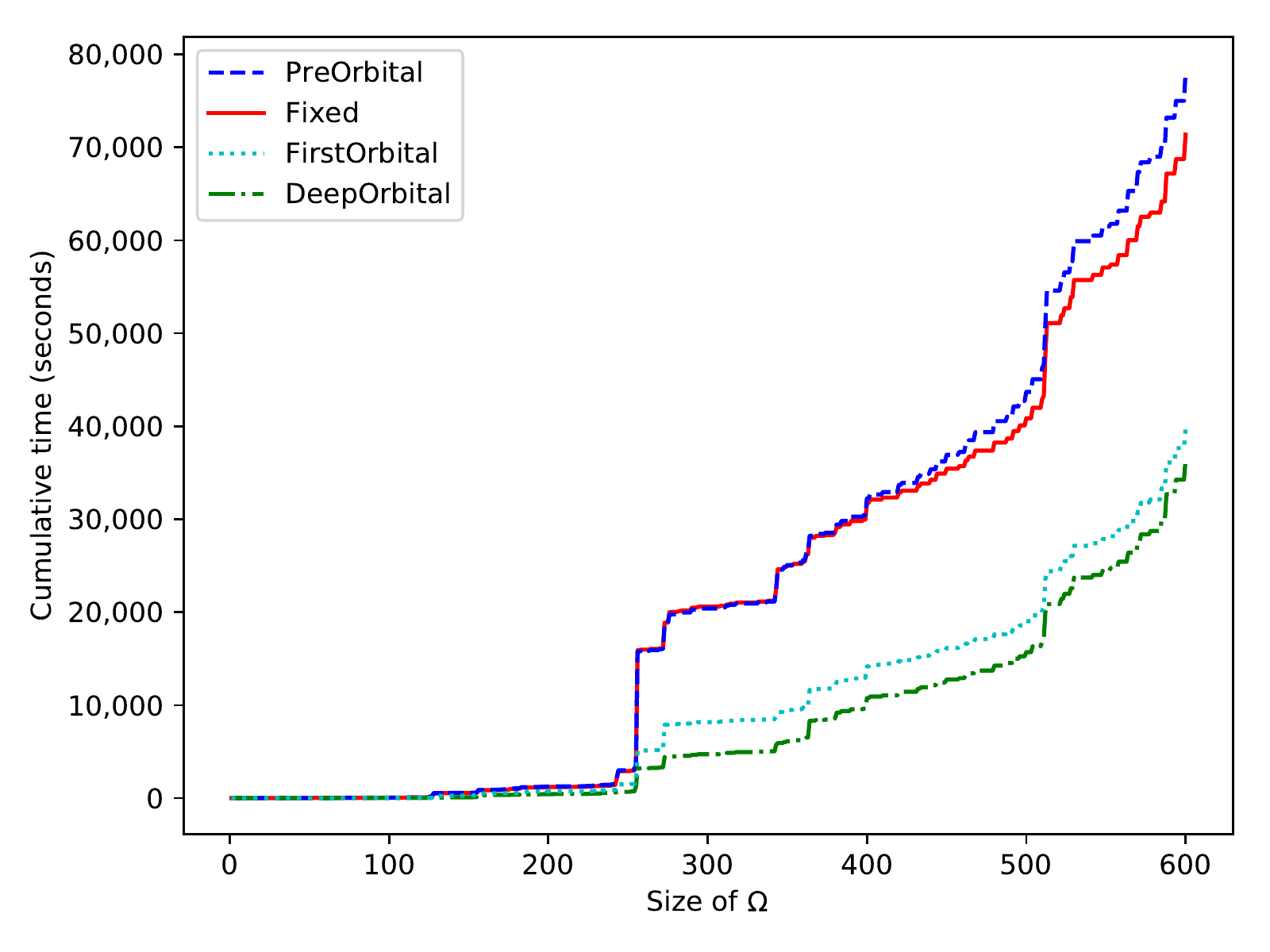}
\end{center}
\caption{Cumulative time taken to intersect 2-transitive groups with a wreath product\label{fig:graph2prim2}}
\end{figure}

\section{Final comments and future work}

Our experiments show that orbital graphs can be very useful for refining ordered partitions. On a large range of problems they provide significant performance improvements, and the worst-case slowdowns they cause are not significantly harmful. In particular, \textbf{FirstOrbital} provides a simple and cheap method of significantly improving the state of the art. However, we cannot theoretically predict in advance when orbital graphs will be useful, and exactly how useful they will be: How can we measure the usefulness of a graph for refinements? Can we always find a graph that is useful in terms of refining ordered partitions?

The results in Section 4 suggest that it is worth investigating alternative refiners in more detail -- if it is possible to efficiently feed structural properties of groups or their action under consideration, into a refiner, then we may be able to improve the runtime of the search algorithm even more.
In cases where using $\Orb$ and $\DeepOrb$ still leads to long run-times in practice, we intend to investigate other types of graphs and combinatorial structures that can be used for refinements in future work. 
Also we plan to improve other parts of the partition backtrack framework, because for many problems we are reaching the limit of improvements that can be provided by better refiners alone. This is because the search involves almost no branching anymore.
In particular, our experiments suggest that
for more substantial progress we must either speed up the calculation of stabilizer chains, or reduce the number of stabilizer chains that must be found.

\bibliographystyle{elsarticle-harv}
\bibliography{GraphsReallyHelp}{}
\end{document}